\documentclass[12pt,reqno]{amsart}
\usepackage{latexsym, amsmath, amssymb,amsthm,amsopn,amsfonts}
\usepackage{version}
\usepackage{epsfig,graphics,color,graphicx,graphpap}
\usepackage{amssymb}
\usepackage{xcolor}
\usepackage[hidelinks]{hyperref}
\hypersetup{
    colorlinks,
    linkcolor={blue!80!black},
    citecolor={blue!80!black},
    urlcolor={blue!80!black}
}
\ifx\pdfoutput\undefined
  \DeclareGraphicsExtensions{.eps}
\else
  \ifx\pdfoutput\relax
    \DeclareGraphicsExtensions{.eps}
  \else
    \ifnum\pdfoutput>0
      \DeclareGraphicsExtensions{.pdf}
    \else
      \DeclareGraphicsExtensions{.eps}
    \fi
  \fi
\fi

\newcommand{\RR}{{\mathbb R}}

\renewcommand{\Re}{\mathop{\rm Re}\nolimits}
\renewcommand{\Im}{\mathop{\rm Im}\nolimits}

\newcommand{\norm}[1]{\left\Vert#1\right\Vert}
\newcommand{\abs}[1]{\left\lvert#1\right\rvert}

\theoremstyle{plain}

\newtheorem{theorem}{Theorem}
\newtheorem{proposition}{Proposition}[section]
\newtheorem{cor}[proposition]{Corollary}
\newtheorem{lemma}[proposition]{Lemma}
\newtheorem{rmk}[proposition]{Remark}

\theoremstyle{definition}

\numberwithin{equation}{section}
\newcommand{\thmref}[1]{Theorem~\ref{#1}}
\newcommand{\thmsref}[2]{Theorems~\ref{#1} and \ref{#2}}
\newcommand{\secref}[1]{Section~\ref{#1}}
\newcommand{\lemref}[1]{Lemma~\ref{#1}}

\newcommand{\propref}[1]{Proposition~\ref{#1}}

\def\squarebox#1{\hbox to #1{\hfill\vbox to #1{\vfill}}}

\newcommand{\la}{\langle}
\newcommand{\ra}{\rangle}

\usepackage{amsxtra}

\newcommand{\R}{\mathbb{R}}
\newcommand{\Z}{\mathbb{Z}}
\newcommand{\C}{\mathbb{C}}
\newcommand{\ve}{\varepsilon}

\newcommand{\Imag}{\operatorname{Im\,}}
\newcommand{\lp}{\left(}
\newcommand{\rp}{\right)}

\newcommand{\wt}{\widetilde}
\newcommand{\Hone}{H^{(1)}}
\newcommand{\Htwo}{H^{(2)}}
\newcommand{\II}{I\!\!I}
\newcommand{\lms}{L^{2,-\sigma}}
\newcommand{\ls}{L^{2,\sigma}}
\newcommand{\lpn}[2]{ \norm{#1}_{L^{#2}}}
\newcommand{\ltwo}[1]{\lpn{#1}{2}}

\definecolor{bpurple}{RGB}{200,50,200}

\usepackage{dsfont}
\usepackage{mathrsfs}

\title
[Dispersive estimates on product cones]
{Pointwise dispersive estimates for Schr\"odinger operators on product cones}  

\author[B. Keeler]{Blake Keeler}
\email{\href{mailto:bkeeler2015@gmail.com}{bkeeler2015@gmail.com}}
\address{Department of Mathematics and Statistics, McGill University\\ Montreal, QC}

\author[J.L. Marzuola]{Jeremy L. Marzuola}
\email{\href{mailto:marzuola@math.unc.edu}{marzuola@math.unc.edu}}
\address{Department of Mathematics, UNC-Chapel Hill \\ CB\#3250
  Phillips Hall \\ Chapel Hill, NC 27599}

\begin{document}

\begin{abstract}
In this manuscript, we investigate the dispersive properties of solutions to the Schr\"odinger equation with a weakly decaying radial potential on cones. If the potential has sufficient polynomial decay at infinity, we obtain a variety of results on the perturbed conic resolvent operator $R_V$ and the nature of the continuous spectrum of $-\Delta +V$. Using these results, we are able to show that the Schr\"odinger flow on each eigenspace of the link manifold satisfies a weighted $L^1\to L^\infty$ dispersive estimate. In odd dimensions, the decay rate we compute is consistent with that of the Schr\"odinger equation in a Euclidean space of the same dimension, but the spatial weights reflect the more complicated regularity issues in frequency that we face in the form of the spectral measure. In even dimensions, we prove a similar estimate, but with a loss of $t^{1/2}$ compared to the sharp Euclidean estimate. 
\end{abstract}
   
\maketitle 

\section{Introduction}
\label{intro}

Let $(X,h)$ be a smooth, compact Riemannian manifold of dimension $n-1$, and consider the cone on $X$, denoted $C(X)$ and defined as $\R^+\times X$ with metric $g$ given by $g = dr^2 + r^2h.$ The corresponding Laplace operator on $C(X)$ is given by
\[\Delta_{C(X)} = \partial_r^2 + \frac{n-1}{r}\partial_r + \frac{1}{r^2}\Delta_h,\]
where $\Delta_h$ is the Laplacian on $X$, taken with the negative semidefinite sign convention. We take $\Delta_{C(X)}$ with the Friedrich's extension for simplicity. We are interested in dispersive estimates for the Schr\"odinger flow 
\begin{equation}\label{schrodinger_flow}
e^{itH}P_c,\quad H = -\Delta_{C(X)} + V,
\end{equation}
where $P_c$ denotes projection onto the continuous spectrum of $H.$ Here, we assume that $V$ is a real-valued radial potential satisfying certain decay assumptions at infinity.

Besides giving direct insight into the behavior of waves, dispersive bounds also have interesting applications in nonlinear problems. For example, stability questions around static solutions in nonlinear models such as wave maps have been studied using dispersive decay estimates. See the work of Krieger-Schlag \cite{krieger2008renormalization} and more recently Krieger-Miao-Schlag \cite{krieger2020stability} for instance.   See also the many works of Lawrie-Oh-Shahshahani \cite{lawrie2018local,lawrie2019asymptotic,lawrie2016cauchy,lawrie2016gap,lawrie2016profile,lawrie2017stability} for treatment of geometric wave and Schr\"odinger equations in hyperbolic space.  Pointwise decay estimates also play a role in obtaining enhanced existence times using normal form methods, see for instance recent works of Ifrim-Tataru \cite{ifrim2015global} and Germain-Pusateri-Rousset \cite{germain2018nonlinear}.   It is also an intrinsically interesting question to understand the interaction between a background potential and diffraction in order to better characterize the dynamics of waves on manifolds with conic singularities. Conic manifolds have arisen naturally in the work of Hintz-Vasy and Hafner-Hintz-Vasy on general relativity, see \cite{hafner2019linear,hintz2015semilinear,hintz2018global} and in particular the recent discussion in the work of Hintz \cite{hintz2020resolvents}.

Dispersive behavior of Schr\"odinger flows has been studied in a tremendous variety of geometric settings and under many different conditions on the asymptotic decay and regularity properties of the potential $V$. In $\R^n$, some of the first ideas arose in the seminal paper of Journ\'e-Soffer-Sogge \cite{journe1991decay}, who proved dispersive decay for $n \geq 3$ with potentials that had no zero energy eigenvalues or resonances and were somewhat strongly decaying and regular.  Since then, decay estimates have been improved in a variety of settings.  Early works by Goldberg and collaborators carefully addressed the regularity required of the potential in higher dimensions and decay rates in $3$ dimensions in the absence of embedded resonance and eigenvalues, see \cite{beceanu2012schrodinger,goldberg2006dispersive,goldberg2006dispersivergh,goldberg2004dispersive,goldberg2006counterexample}.  

Further works for perturbations of the Euclidean Laplacian have extended dispersive decay results to the setting where $-\Delta +V$ has an embedded resonance at zero energy, which results in a weaker decay estimate in time, see for instance especially the works of Erdogan-Schlag in $3$ dimensions
\cite{erdougan2010dispersive,ES1}, Erdogan-Green in two dimensions \cite{erdougan2013dispersive,erdougan2013weighted}, Green in $5$ dimensions \cite{green2012dispersive}, as well as Goldberg-Green and Erdogan-Goldberg-Green in odd and even dimensions $\geq 4$ \cite{erdougan2014dispersive,goldberg2014dispersive,goldberg2015dispersive,goldberg2016lp}.  Recent progress by Blair-Sire-Sogge \cite{blair2019quasimode} has pushed the construction of the spectral measure for $-\Delta + V$ to cases where the regularity of the potential $V$ is at very critical levels, though the authors have not explored dispersive decay directly. This is by no means an exhaustive list, but these results are representative of the techniques involved, namely careful control of the free resolvent, the use of resolvent expansions, the role of the regularity of the potential $V$, and the spectral structure of the operator $-\Delta + V$.  The survey article by Wilhelm Schlag \cite{schlag2007dispersive} contains an excellent overview of the key ideas involved.  

Dispersive decay estimates have also been studied in several other geometries. For example, Schr\"odinger operators with potential were studied on hyperbolic space by David Borthwick and the second author in \cite{borthwick2015dispersive}.  See the recent article of Bouclet \cite{bouclet2018sharp} for a broad overview of results on the asymptotically Euclidean setting, the article by Hassell-Zhang \cite{hassell2016global} and references therein for results on asymptotically conic manifolds, as well as the articles of Schlag-Soffer-Staubach \cite{schlag2010decay1,schlag2010decay2} for manifolds with conical ends.  Analysis of the Laplacian on product cones is related to the analysis of Schr\"odinger operators on $\mathbb{R}^n$ with an inverse square potential which have been studied in various settings, e.g. the works \cite{killip2015energy,miao2013strichartz,miao2014maximal,mizutani2019uniform,zhang2014scattering} by various authors.

The study of the Laplacian on product cones has a rich history.  See the classical results of Cheeger-Taylor, 
\cite{Cheeger1979,cheeger1982diffraction}, where the spectral measure was first described.  As a result, there have been several works that studied evolution equations and their decay estimates on product cones, especially wave equations \cite{baskin2019scattering,blair2012strichartz,blair2013strichartz,Ford2010,ford2015wave,melrose2004propagation,zhang2018strichartz}. See also \cite{baskin2016locating} for information about scattering resonances on hyperbolic cones. 

Analysis of dispersive estimates for Schr\"odinger equations using the resolvent and spectral measure on a product cone has been studied in the recent results of Zhang-Zheng \cite{zhang2017globalcone,zhang2017global}.  These are the most closely related results to ours, but only study specific types of potentials that can be treated more perturbatively, hence they need not fully explore the regularity and decay of the spectral measure in the same fashion undertaken here.  See also the very recent work of Chen \cite{chen2020semiclassical} that studies local dispersive behavior on manifolds with non-product conic singularities.

On pure product cones, we prove pointwise decay estimates for the mode-by-mode decomposition of the Schr\"odinger flow \eqref{schrodinger_flow}. By this, we mean that if $\{\varphi_j\}_{j=0}^\infty$ is a basis of $L^2(X)$ consisting of eigenfunctions of $\Delta_h$, then the Schr\"odinger flow on $C(X)$ can be formally decomposed as 
\begin{equation}\label{modebymode}
e^{itH}P_c = \sum\limits_{j=0}^\infty e^{it H}P_c E_j,
\end{equation}
where $E_j:L^2(X)\to L^2(X)$ denotes projection on to the linear span of $\varphi_j.$ We show that if $V\in \rho^{-2\sigma} L^\infty(\R^+)$ for $\sigma$ sufficiently large, where $\rho(r) = 1 + r$ is a weight function, and if the perturbed resolvent 
\[R_V(z^2):= (-\Delta_{C(X)} +V - z^2)^{-1}\]
does not have a pole at $z = 0$, then each component of \eqref{modebymode} satisfies a weighted pointwise estimate. In odd dimensions, we prove this with the same $t^{-\tfrac{n}{2}}$ decay rate as in the Euclidean case, while in even dimensions, there is a loss of $t^{\tfrac{1}{2}}$ which we do not expect to be sharp. The significance of the resolvent can be seen quite directly if we express the Schr\"odinger flow in terms of the continuous part of spectral measure for $-\Delta_{C(X)} + V$, which we denote by $d\Pi_V$. In particular, if we assume for the moment that there are no resonances or eigenvalues embedded in the continuous spectrum, then we have
\[e^{it H}P_c = \int\limits_0^\infty e^{it\mu}\,d\Pi_V(\mu).\]
By Stone's formula and a change of variables from $\mu$ to $\lambda^2$, we can rewrite the spectral measure in terms of the boundary values of the resolvent via 
\[d\Pi_V(\lambda) = \frac{\lambda}{2\pi i}\left[R_V(\lambda^2+ i0) - R_V(\lambda^2-i0)\right]\,d\lambda.\]
The behavior of the resolvent is thus of critical importance for understanding the properties of the Schr\"odinger flow. Hence, a large portion of this manuscript is dedicated to analyzing the structure of $R_V(z^2)$, or more specifically its projections $R_{V,j}(z^2) = R_V(z^2)E_j.$ Our first main result establishes that each $R_{V,j}$ admits a meromorphic continuation to the logarithmic cover of $\C\setminus\{0\}$ and satisfies a version of the limiting absorption principle.

\begin{theorem}\label{LAP}
For $V\in \rho^{-2\sigma}L^{\infty}(\R^+)$ with $\sigma > \frac{1}{2},$ $R_{V,j}(z^2)$ admits a meromorphic continuation to the logarithmic cover of $\C\setminus\{0\}$. In the region where $\Imag(z^2) > 0$, we have that 
\[R_{V,j}(z^2):L^{2,\delta}(\R^+,r^{n-1}\,dr)\to L^{2,-\delta}(\R^+,r^{n-1}\,dr)\]
is a bounded operator for all $\frac{1}{2} < \delta < \sigma.$ Here, the notation $L^{2,\delta}(\R^+,r^{n-1}\,dr)$ denotes the weighted $L^2$-space given by $\{f:\R^+\to\C: \int |f(r)|^2\rho(r)^{2\delta}\,r^{n-1}\,dr < \infty\}.$ 

Furthermore, if $\sigma > \frac{1}{2} + k$, then the derivatives of $R_{V,j}$ up to order $k$ satisfy the limiting absorption principle. That is, for $0 \le \ell \le k$, there exists an $M_V > 0$ such that
\begin{equation}
\partial_\lambda^\ell R_{V,j}(\lambda^2 \pm i0): L^{2,\delta}(\R^+,r^{n-1}\,dr)\to L^{2,-\delta}(\R^+,r^{n-1}\,dr)
\end{equation}
is a bounded operator for all $\lambda \ge M_V$ and all $\frac{1}{2} + k <\delta < \sigma$. In particular, we have the operator bound 
\begin{equation}\label{LAP_eqn}
\left\|\partial_\lambda^\ell R_{V,j}(\lambda^2 \pm i0)\right\|_{L^{2,\delta}\to L^{2,-\delta}}\le \frac{C_{j,\ell}}{\lambda}
\end{equation}
for each $0 \le \ell \le k, $ for some $C_{j,\ell} > 0$ and all $\lambda \ge M_V .$
\end{theorem}

\noindent We note that \eqref{LAP_eqn} implies that $-\Delta_{C(X)}+V$ does not have any embedded resonances in the interval $[M_V,\infty)$, but analysis of the Schr\"odinger flow requires information about the spectrum at low energy as well. The next theorem handles this by showing that indeed there are no embedded eigenvalues or resonances in $(0,\infty).$ In fact, for this theorem, we do not need $V$ to be a radial potential, since the arguments involved do not rely as heavily on the conic structure.

\begin{theorem}\label{absence.thm}
For $V\in \rho^{-2\sigma}L^{\infty}(C(X))$ with $\sigma > \frac{1}{2}$, then $-\Delta_{C(X)} + V$ has continuous spectrum $[0,\infty)$, with no embedded eigenvalues or resonances in the range $(0,\infty)$.
\end{theorem}

\noindent This theorem will be used implicitly throughout this manuscript, but we postpone its proof until Appendix \ref{app:embres}, since the techniques involved in the proof are quite distinct from those used elsewhere in the argument.

In order to analyze the Schr\"odinger flow, we also require estimates on the behavior of the resolvent at low energy. For this we must assume that $-\Delta_{C(X)} + V$ does not have a resonance at zero. With this assumption, we obtain the following refinements in our operator bounds from \thmref{LAP}.

\begin{theorem}\label{low_energy_thm}
Suppose $V\in\rho^{-2\sigma}L^{\infty}(\R^+)$ with $\sigma > \frac{1}{2} + k$, and assume that $-\Delta_{C(X)} + V$ does not have a resonance at zero energy. Then, \eqref{LAP_eqn} can be improved to 
\begin{equation}\label{eqn:RV_L2bd}
\left\|\partial_\lambda^\ell R_{V,j}(\lambda^2 \pm i0)\right\|_{L^{2,\delta}\to L^{2,-\delta}} \le \frac{C_{j,\ell}}{\langle \lambda\rangle} \quad \text{for all }\lambda \ge 0
\end{equation}
for $\frac{1}{2} + k < \delta < \sigma$ and $0\le \ell\le k$. Under the stronger hypothesis that $\frac{n}{2} + k\le \delta < \sigma$, we also have that the imaginary part of $R_{V,j}$ satisfies
\begin{equation}\label{RV_imaginarypart}
\left\|\partial_\lambda^\ell \Imag R_{V,j}(\lambda^2 \pm i0)\right\|_{L^{2,\delta}\to L^{2,-\delta}} \le C_{j,\ell}' \lambda^{n-2-k} \quad \text{for all }\lambda\in [0,1],
\end{equation}
for some $C_{j,\ell}' > 0$ and each $0\le \ell\le k.$
\end{theorem}

\noindent By combining these mapping properties with the behavior of the free resolvent $R_0(z^2): = (-\Delta_{C(X)} - z^2)^{-1}$, we are able to establish weighted pointwise estimates on the Schwartz kernel of $R_{V,j}(\lambda^2 \pm i0)$, from which we can obtain our primary result on the time-decay rate of the Schr\"odinger flow.

\begin{theorem}\label{disp_est}
Suppose $C(X)$ is of odd dimension $n\ge 3$. Let $V\in \rho^{-2\sigma}L^\infty(\R^+) $ with
\[\sigma > 2n\left\lceil\frac{n}{4}\right\rceil.\]
 If $R_V(z^2)$ does not have a pole at $z = 0$, then for any integer $j \geq 0$ and  $\alpha \ge 2\left\lceil\frac{n}{4}\right\rceil(n-2)- \frac{n-1}{2}+2$, we have 
\begin{equation}\label{main_disp_est}
\|\rho^{-\alpha} e^{itH}P_c E_jf\|_{L^\infty(\R^+)} \le C_{j,\alpha,\sigma} t^{-\frac{n}{2}}\|\rho^{\alpha}E_jf\|_{L^1(\R^+,r^{n-1}\,dr)},
\end{equation}
for some $C_{j,\alpha,\sigma} > 0.$
\end{theorem}

We do not claim that the lower bound on the exponent $\alpha$ in the spatial weights is optimal, but these weights are required to obtain \thmref{disp_est} from the techniques used in this article. In particular, the weights are needed to counteract certain regularity issues which arise when differentiating the resolvent with respect to $\lambda$. See Remark \ref{optrmk} for additional details. Furthermore, the dependence on $j$ appears as a consequence of the fact that the pointwise bounds we establish on each $R_{V,j}$ are not necessarily summable in $j$. Similar weighted mode-by-mode estimates are obtained in the works of Schlag-Soffer-Staubach \cite{schlag2010decay1,schlag2010decay2} in the case of surfaces of revolution and related mode by mode decay rates were established for the wave equation on the Schwarzschild space-time in Donninger-Schlag-Soffer \cite{donninger2012pointwise}.

\begin{rmk}\textnormal{In the case where $n$ is even, the techniques of this article give a slightly weaker estimate of the form 
\begin{equation}\label{disp_est_even}
\|\rho^{-\alpha} e^{itH}P_c E_jf\|_{L^\infty(\R^+)} \le C_{j,\alpha,\sigma} t^{-\frac{n-1}{2}}\|\rho^{\alpha}E_jf\|_{L^1(\R^+,r^{n-1}\,dr)},
\end{equation}
for analogous conditions on $V$ and $\alpha,$ where the loss of the $\frac{1}{2}$ power of decay in $t$ arises as a result of regularity issues encountered in the analysis of the spectral measure near zero energy. We expect that with more sophisticated techniques it may be possible to improve this estimate to give the full $t^{-\frac{n}{2}}$ decay rate exhibited in $\R^n$. }
\end{rmk}

\subsection{Outline of the Paper}

In Section \ref{free}, we summarize known facts regarding the form of the free resolvent $R_0(z^2)$. The results presented are primarily taken from \cite{baskin2019scattering}, which draws heavily upon the seminal work of \cite{cheeger1982diffraction}. In Section \ref{R0_estimates}, we prove operator bounds on the free resolvent using properties of Bessel and Hankel functions. These bounds include a limiting absorption principle for the projections $R_{0,j}:= R_{0}E_j$. Section \ref{spectheory} combines these bounds on the free resolvent with perturbation theory arguments to prove \thmsref{LAP}{low_energy_thm}. Then, in Section \ref{spectralresolution}, we prove weighted pointwise bounds on the perturbed resolvent kernel using a Birman-Schwinger expansion along with the previous operator estimates.  Finally, in Section \ref{dispersive}, we use the representation of the spectral measure in terms of the resolvent combined with the pointwise resolvent bounds to establish Theorem \ref{disp_est}. We also provide three appendices at the end of the paper. Appendix \ref{app:embres} contains the proof of \thmref{absence.thm}, which demonstrates the absence of embedded eigenvalues and resonances. This fact is of critical importance throughout the paper. For the benefit of the reader, Appendix \ref{app:free_res} provides a full derivation of the free resolvent formula presented in \secref{free}. This derivation largely follows the work of \cite{baskin2019scattering}, but we provide some additional clarifying details. Finally, Appendix \ref{ell2_dispersive} uses ideas from \cite{Ford2010} to give a modified dispersive estimate for the free Schr\"odinger flow which is both unweighted and not restricted to individual eigenspaces of the link manifold. 

\subsection*{Acknowledgements}  BK is supported by DMS 1900519 and Sloan Fellowship through his advisor Yaiza Canzani.  JLM was supported in part by NSF CAREER Grant DMS-1352353 and NSF Applied Math Grant DMS-1909035.  JLM also thanks Duke University and MSRI for hosting him during the outset of this research project. The authors would like to thank Dean Baskin, David Borthwick, Yaiza Canzani, Michael Goldberg, Andrew Hassell and Jason Metcalfe for very helpful discussions about resolvent estimates on conic manifolds and pointwise estimates in general.   The authors thank the anonymous referee who reviewed the first version of this manuscript for many valuable comments that led a helpful reorganization of the results and several aspects of the main theorems being clarified.

\section{The Free Resolvent}
\label{free}

\noindent In this section, we outline some key facts about the kernel of the free resolvent operator
\begin{equation}
R_0(z^2) = (-\Delta_{C(X)} - z^2)^{-1}: L^2(C(X))\to L^2(C(X)),
\end{equation}
for complex $z$. This is equivalent to analyzing solutions of the equation
\begin{equation}\label{free_resolvent}
(-\Delta_{C(X)} - z^2)u = f
\end{equation}
for $f\in L^2(C(X)).$ To proceed, we decompose $u$ and $f$ into the basis $\{\varphi_j\}$ of eigenfunctions on $X$ as
$$f(r,\theta) = \sum\limits_{j=1}^\infty f_j(r)\varphi_j(\theta), \quad u(r,\theta) = \sum\limits_{j=1}^\infty u_j(r)\varphi_j(\theta).$$ 
Denote by $-\mu_j^2$ the eigenvalues of $\Delta_h$ associated to each $\varphi_j$. Then, we obtain that \eqref{free_resolvent} is equivalent to the collection of equations
\begin{equation}\lp\partial_r^2 + \frac{n-1}{r}\partial_r +z^2 - \frac{\mu_j^2}{r^2}\rp u_j(r) = -f_j(r) ,\hskip 0.2in j = 0,1,2,\dotsc.\label{jth_resolvent_eqn}
\end{equation}
Therefore, we can express the resolvent $R_0(z^2)$ as 
\[R_0(z^2)f(r,\theta) = \sum\limits_{j=0}^\infty u_j(r)\varphi_j(\theta),\]
with $u_j$ as above. If we define the $j$th \textit{radial resolvent} $R_{0,j}(z^2)$ by
\begin{equation}\label{radial_resolvents}
R_{0,j}(z^2) = \lp\partial_r^2 + \frac{n-1}{r}\partial_r + z^2 - \frac{\mu_j^2}{r^2}\rp^{-1}
\end{equation}
as an operator on $L^2(\R^+,r^{n-1}\,dr)$, then the full resolvent is given by
\[R_0(z^2)f(r,\theta) = \sum\limits_{j=0}^\infty R_{0,j}(z^2)f_j(r)\varphi_j(\theta).\]
\noindent The work of Baskin and Yang \cite{baskin2019scattering} presents several results about these radial resolvents, which we summarize in the following lemma.

\begin{lemma}[Baskin-Yang \cite{baskin2019scattering}]
\label{baskin_yang_lemma}
For $\Imag z > 0$, the action of the $j$-th radial resolvent is given by
\[R_{0,j}(z^2)f(r) = \int\limits_0^\infty R_{0,j}(z^2;r;s)f(s)s^{n-1}\,ds,\]
and the kernel $R_{0,j}(z^2;r;s)$ takes the form
\begin{equation}\label{radial_resolvent_kernel}
R_{0,j}(z^2;r;s) = \begin{cases}\frac{\pi i}{2}(rs)^{-\frac{n-2}{2}}J_{\nu_j}(z s)H_{\nu_j}^{(1)}(z r), & s < r\\
\frac{\pi i}{2}(rs)^{-\frac{n-2}{2}}J_{\nu_j}(z r)H_{\nu_j}^{(1)}(z s), & s > r,
\end{cases}
\end{equation}
where $J_{\nu_j}$ and $H_{\nu_{j}}^{(1)}$ denote the Bessel and Hankel functions of the first kind of order $\nu_j$, respectively. Moreover, for any fixed $\chi\in C_c^\infty(\R^+\times X)$, the cutoff resolvent $\chi R_0(z^2)\chi$ admits a meromorphic continuation to the logarithmic cover $\Lambda$ of $\C\setminus 0$. 
\end{lemma}
\begin{rmk}\textnormal{ We note that the above formula for the kernel of $R_{0,j}$ differs from that presented in \cite{baskin2019scattering} by a sign, since we have defined the resolvent as $(-\Delta_{C(X)} - z^2)^{-1}$ rather than $(\Delta_{C(X)} + z^2)^{-1}$, but this is of no consequence for the remainder of the analysis. }
\end{rmk}

From \lemref{baskin_yang_lemma}, we can construct the absolutely continuous part of the spectral measure for the free Laplacian on $C(X)$, which we denote by $d\Pi_0$. By Stone's formula, we can write the continuous part of the spectral measure in terms of the difference between the boundary values of the resolvent as we approach the continuous spectrum from above and below. That is, for $\mu \in \R^+,$
\begin{align*}
d\Pi_0(\mu) &= \frac{1}{2\pi i}\lim\limits_{\ve\to 0^+}\left[(-\Delta_{C(X)} - (\mu + i\ve)^{-1} - (-\Delta_{C(X)} - (\mu - i\ve))^{-1}\right]\,d\mu\\
& = \frac{1}{2\pi i}\Im (-\Delta_{C(X)} - (\mu + i0))^{-1}\,d\mu.
\end{align*}
We can then reparametrize the continuous spectrum by changing variables via $\mu \mapsto \lambda^2$ for $\lambda > 0$, which allows us to write 
\[d\Pi_0(\mu) = \frac{1}{\pi i}\Im R_0(\lambda^2+i0)\lambda\,d\lambda.\]
Noting that $H_{\nu}^{(1)} = J_{\nu_j}+iY_{\nu}$, where $Y_{\nu}$ is the Bessel function of the second kind of order $\nu,$ we have by \lemref{baskin_yang_lemma} that
\[\Im R_{0,j}(\lambda^2+i0;r,s) = \frac{\pi}{2}(rs)^{-\frac{n-2}{2}}J_{\nu_j}(\lambda r)J_{\nu_j}(\lambda s).\]
From this, we obtain the following lemma.

\begin{lemma}\label{spectral_measure_free}
The continuous part of the spectral measure of $-\Delta_{C(X)}$, with the convention that $\lambda^2$ is the spectral parameter, is given by
\[d\Pi_0(\lambda;x,y) = \frac{1}{\pi i}(rs)^{-\frac{n-2}{2}}\sum\limits_{j=0}^\infty J_{\nu_j}(\lambda r)J_{\nu_j}(\lambda s)\varphi_j(\theta)\overline{\varphi_j(\zeta)}\,\lambda\,d\lambda, \quad \lambda > 0,\]
where $x = (r,\theta)$ and $y = (s,\zeta)$ are points in $C(X)$.
\end{lemma}

\section{Estimates on the Free Resolvent}
\label{R0_estimates}

\noindent In this section, we prove a variety of weighted estimates on the unperturbed radial resolvents ${R_{0,j}}$. These estimates heavily rely on the asymptotic formulae for the Bessel and Hankel functions near zero and infinity. Of particular interest is the behavior of $R_{0,j}$ measured in the weighted $L^q$ spaces defined by
\[
L^{q,\sigma}(\R^+,\,r^{n-1}\,dr) = \{f:\R^+\to\C:\, \int_0^\infty\left|f(r)\right|^q\rho^{q\sigma}(r)\,r^{n-1}\,dr < \infty\},
\]
where $\rho(r) = 1+r.$ For ease of notation, we simply write $L^{q,\sigma}$ to denote the space $L^{q,\sigma}(\R^+,\,r^{n-1}\,dr)$ where there can be no confusion. The estimates for the free resolvent on these spaces will prove useful in Sections \ref{spectheory} and \ref{spectralresolution} for establishing the mapping properties of the perturbed resolvent. 

We begin with a quantitative formulation of the Limiting Absorption Principle for the radial resolvents.  See \cite{vasy2021limiting} for a recent discussion of this in the more general setting of scattering manifolds.  

\begin{proposition}
\label{L2_L2_prop}
\noindent Let $k\ge 0$ be an integer. Then for any $\sigma > \frac{1}{2} + k$, 
\begin{equation}
\|\partial_\lambda^k R_{0,j}(\lambda^2+i0)\|_{L^{2,\sigma}\to L^{2,-\sigma}} \le \frac{C_{j,k,\sigma}}{|\lambda|}
\end{equation}
for all $|\lambda| \ge 1.$ 
\end{proposition} 

\begin{rmk}\textnormal{
A noteworthy observation here is that the constant $C_{j,k,\sigma}$ in \propref{L2_L2_prop} is not known \emph{a priori} to be bounded as a function of $j$. In the special case where $k = 0$, the statement of \propref{L2_L2_prop} can be shown to hold for the full resolvent $R_0(\lambda^2+i0)$ with a uniform constant using extremely precise asymptotics for the Bessel and Hankel functions such as those found in \cite{FrankSimon2017}. However, when $k > 0$ this method fails due to the fact that differentiating $J_\nu(\lambda r)H_\nu^{(1)}(\lambda s)$ yields a linear combination of products of Bessel and Hankel functions with mismatched orders, and hence the resulting constants in the estimates for the Hankel functions are not balanced by those of the Bessel functions, in contrast to the $k= 0$ case. 
}
\end{rmk}

\begin{proof}
If $f\in L^{2,\sigma}$, we have
\begin{equation*}\|\partial_\lambda^kR_{0,j}(\lambda^2+i0)f\|_{\lms}^2 = \int\limits_0^\infty \left|\int\limits_0^\infty \partial_\lambda^k R_{0,j}(\lambda^2+i0;r,s)f(s) s^{n-1}\,ds\right|^2 (1+r)^{-2\sigma}r^{n-1}\,dr.\end{equation*}
Inserting a factor of $(1+s)^{-\sigma}(1+s)^{\sigma}$ and applying Cauchy-Schwartz, we see that

\begin{align*}
\begin{split}
\|\partial_\lambda^kR_{0,j}(\lambda^2+i0)f\|_{\lms}^2 & \le \int\limits_0^\infty \|\partial_\lambda^kR_{0,j}(\lambda^2+i0;r,\cdot)\|_{\lms_s}^2\|f\|^2_{\ls_s}(1+r)^{-2\sigma}r^{n-1}\,dr\\
& = \|\partial_\lambda^kR_{0,j}(\lambda^2+i0;\cdot,\cdot)\|^2_{L^{2,-\sigma}_s L^{2,-\sigma}_r}\|f\|^2_{L^{2,\sigma}_s}.
\end{split}
\end{align*}
Hence, it suffices to show that the kernel satisfies
\[
\|\partial_\lambda^k R_{0,j}(\lambda^2+i0;\cdot,\cdot)\|^2_{L^{2,-\sigma}_s L^{2,-\sigma}_r} \le \frac{C}{\lambda^2}.
\]
By definition,
\begin{equation}
\|\partial_\lambda^kR_{0,j}(\lambda^2+i0;\cdot,\cdot)\|^2_{L^{2,-\sigma}_s L^{2,-\sigma}_r} = \int\limits_0^\infty\int\limits_0^\infty \partial_\lambda^k R_{0,j}(\lambda^2+i0;r,s)(1+s)^{-2\sigma}(1+r)^{-2\sigma}(rs)^{n-1}\,ds\,dr,
\end{equation}
Recalling the piecewise formula \eqref{radial_resolvents_plus} for the resolvent kernel, we have that
\begin{align}\label{weighted_norm}
\begin{split}
\|\partial_\lambda^kR_{0,j}(\lambda^2+i0;\cdot,\cdot)\|^2_{\lms_s\lms_r}  &\\
& \hspace{-1in} =\frac{\pi^2}{4}\int\limits_0^\infty \int\limits_0^r\left[\partial_\lambda^k\lp H_{\nu_j}^{(1)}(\lambda r) J_{\nu_j}(\lambda s)\rp\right]^2(rs)(1 + r)^{-2\sigma}(1+s)^{-2\sigma}\,ds\,dr\\
& \hspace{-.9in} + \frac{\pi^2}{4}\int\limits_0^\infty \int\limits_r^\infty \left[\partial_\lambda^k\lp H_{\nu_j}^{(1)}(\lambda s) J_{\nu_j}(\lambda r)\rp\right]^2(rs)(1 + r)^{-2\sigma}(1+s)^{-2\sigma}\,ds\,dr . 
\end{split}
\end{align}
By changing the order of integration, we get that the first term on the right-hand side of \eqref{weighted_norm} can be rewritten as
\begin{equation}
\frac{\pi^2}{4}\int\limits_0^\infty \int\limits_s^\infty \left[\partial_\lambda^k\lp H_{\nu_j}^{(1)}(\lambda r) J_{\nu_j}(\lambda s)\rp\right]^2(rs)(1 + r)^{-2\sigma}(1+s)^{-2\sigma}\,dr\,ds.
\end{equation}
We note that up to a relabeling of $r,s$, this is exactly equal to the second term in \eqref{weighted_norm}, and hence
\begin{equation}
\label{plamR0bd}
\begin{split}
& \|\partial_\lambda^kR_{0,j}(\lambda^2+i0;\cdot,\cdot)\|^2_{\lms_s\lms_r} = \\
& \hspace{1cm} \frac{\pi^2}{2}\int\limits_0^\infty \int\limits_r^\infty \left[\partial_\lambda^k\lp H_{\nu_j}^{(1)}(\lambda s) J_{\nu_j}(\lambda r)\rp\right]^2(rs)(1 + r)^{-2\sigma}(1+s)^{-2\sigma}\,ds\,dr.
\end{split}
\end{equation}

Note that if $\mathcal C_\nu(x)$ is either a Bessel or Hankel function of order $\nu,$ we have 
\begin{equation}\label{Bessel_derivative}
\mathcal C_\nu'(x) = \frac{1}{2}\lp\mathcal C_{\nu+1}(x) - \mathcal C_{\nu-1}(x)\rp,
\end{equation}
and so the triangle inequality reduces the proof of \propref{L2_L2_prop} to showing that the following lemma holds.

\end{proof}

\begin{lemma}\label{L2_L2_lemma}
Let $\ell,m,k$ be nonnegative integers with $\ell+m =k$, and suppose $\alpha,\beta\in\Z$ are such that $|\alpha|\le \ell$ and $|\beta| \le m.$ Then for any $\nu \ge \frac{n-2}{2}$, there exists a $C > 0$ depending only on $k,\nu$ such that
\begin{equation}\label{L2_L2_lem_eqn}
\int\limits_0^\infty\int\limits_r^\infty |J_{\nu+\alpha}(\lambda r)|^2|H^{(1)}_{\nu+\beta}(\lambda s)|^2 r^{1 + 2\ell}s^{1 + 2m}(1 + s)^{-2\sigma}(1 + r)^{-2\sigma}\,ds\,dr \le \frac{C}{\lambda^2},\hskip 0.2in \lambda\ge 1,
\end{equation}
provided that $\sigma > \frac{1}{2}+k$. 
\end{lemma}

\begin{proof}
This proof, and others which follow it, make extensive use of asymptotic estimates for the Bessel and Hankel functions, which we record here for later use. For any $\nu\in\R,$ there exist constants $C_\nu,C_\nu'>0$ such that when $0 < |\tau| \le 1$,
\begin{equation}\label{small_arg}
\left|J_\nu(\tau)\right| \le C_\nu|\tau|^\nu, \quad \left|H_{\nu}^{(1)}(\tau)\right| \le C_\nu |\tau|^{-\nu}
\end{equation}
and when $|\tau| \ge 1$,
\begin{equation}\label{large_arg}
\left|J_\nu(\tau)\right| \le C_\nu'|\tau|^{-\tfrac{1}{2}},\quad \left|H_{\nu}^{(1)}(\tau)\right| \le C_\nu'|\tau|^{-\tfrac{1}{2}}.
\end{equation}

To prove \lemref{L2_L2_lemma}, let us first write the left-hand side of \eqref{L2_L2_lem_eqn} as $I(\lambda)+\II(\lambda)$, where each term is obtained by restricting the integral in the $r$ variable to $0 < r < \frac{1}{\lambda}$ and $\frac{1}{\lambda}< r < \infty$, respectively.  To estimate $I(\lambda)$, note that by \eqref{large_arg}, we have
\[\int\limits_{\frac{1}{\lambda}}^\infty s^{1 + 2m}(1 + s)^{-2\sigma}|H^{(1)}_{\nu+\beta}(\lambda s)|^{2}\,ds \le \frac{C}{\lambda}\int\limits_{\frac{1}{\lambda}}^\infty s^{2m}(1 + s)^{-2\sigma}\,ds \le \frac{C'}{\lambda},\]
as long as $\sigma > m + \frac{1}{2}$. Combining this with \eqref{small_arg}, we have
\begin{align} \label{I_bound_1}
\begin{split}
 \int\limits_0^{\frac{1}{\lambda}}\int\limits_{\frac{1}{\lambda}}^\infty |J_{\nu+\alpha}(\lambda r)|^2|H^{(1)}_{\nu+\beta}(\lambda s)|^2 r^{1 + 2\ell}s^{1 + 2m}(1 + s)^{-2\sigma}(1 + r)^{-2\sigma}\,ds\,dr &\le \frac{C'}{\lambda} \int\limits_{0}^\frac{1}{\lambda}r^{1 + 2\ell}(\lambda r)^{2(\nu+\alpha)}\,dr \\
&\le \frac{C''}{\lambda^2} 
\end{split}
\end{align}
for any $\ell\ge 0,$ since $\lambda \ge 1.$ Furthermore, if $0<r \le \frac{1}{\lambda}$, we have
\begin{align*}
\int\limits_r^{\frac{1}{\lambda}} s^{1 + 2m}(1 + s)^{-2\sigma}|H^{(1)}_{\nu+\beta}(\lambda s)|^{2}\,ds  \le \frac{C}{\lambda^{2(\nu+\beta)}}\int\limits_r^{\frac{1}{\lambda}}s^{1 + 2(m-\nu-\beta)}\,ds\\ 
 \hspace{.5cm}\le \frac{C}{\lambda^{2(\nu+\beta)}}r^{1 + 2(m-\nu-\beta)}\lp\frac{1}{\lambda}-r\rp
\le \frac{C'}{\lambda}(\lambda r)^{-2(\nu+\beta)}r^{2m}
\end{align*}
since $1 + 2(m-\beta -\nu) \le 0.$ Hence, if we recall that $k = \ell + m\ge |\alpha|+|\beta|$ and apply \eqref{small_arg}, we have

\begin{align*}
&\int\limits_0^{\frac{1}{\lambda}}\!\!\!\int\limits_r^{\frac{1}{\lambda}} |J_{\nu+\alpha}(\lambda r)|^2|H^{(1)}_{\nu+\beta}(\lambda s)|^2 r^{1 + 2\ell}s^{1 + 2m}(1 + s)^{-2\sigma}(1 + r)^{-2\sigma}\,ds\,dr\\
&\hspace{1cm} \le \frac{C'}{\lambda}\int\limits_0^{\frac{1}{\lambda}} r^{1 + 2k}(\lambda r)^{2(\alpha-\beta)}dr =\frac{C'}{\lambda^{2k+2}}\int\limits_0^\frac{1}{\lambda}(\lambda r)^{1 + 2(k+\alpha-\beta)}dr\\
&\hspace{2cm}\le \frac{C'}{\lambda^{2k+3}}  \le \frac{C'}{\lambda^2}
\end{align*}
 for all $k\ge 0$ when $\lambda \ge 1.$ Combining this with \eqref{I_bound_1}  proves that $I(\lambda)\le \frac{C}{\lambda^2}$ for some $C>0$ and all $\lambda \ge 1$. Since for any fixed $k$ there are only finitely many possibilities for $\ell,m,\alpha,\beta$, we can choose $C$ to depend only on $k$ and $\nu.$ 

Now, to estimate $\II(\lambda)$, we apply \eqref{large_arg} to both the Bessel and Hankel functions to obtain
\begin{align*}
\II(\lambda) & \le C\int\limits_{\frac{1}{\lambda}}^\infty\int\limits_{\frac{1}{\lambda}}^\infty r^{1+2\ell}(1 + r)^{-2\sigma}(\lambda r)^{-1}s^{1 + 2m}(1 + s)^{-2\sigma}(\lambda s)^{-1}\,ds\,dr \\
& = \frac{C}{\lambda^2}\int\limits_0^\infty\int\limits_0^\infty (1+r)^{2(\ell -\sigma)} (1 + s)^{2(m-\sigma)}\,ds\,dr\\
&\le \frac{C}{\lambda^2},
\end{align*}
provided that $\sigma > k + \frac{1}{2}$, which completes the proof of \lemref{L2_L2_lemma}.
\end{proof}

It will also prove useful to have a bound on the $L^{2,\sigma}\to L^{2,-\sigma}$ mapping properties of the imaginary part of each $R_{0,j}$ when $\lambda$ is small. In particular, we are able to show that this operator norm has a precise polynomial rate of vanishing as $\lambda \to 0.$ 

\begin{proposition}\label{L2_L2_imaginary}
For any integer $k\ge 0$ and any $\sigma > \frac{n}{2} + k $, we have that
\[\|\partial_\lambda^k \Imag R_{0,j}(\lambda^2+i0)\|_{L^{2,\sigma}\to L^{2,-\sigma}}\le C_{j,k,\sigma}\lambda^{n-2-k}\]
when $0 < \lambda \le 1.$
\end{proposition}

\begin{proof}
By the discussion at the beginning of the proof of \propref{L2_L2_prop}, it is sufficient to show that 
\begin{equation}\label{imag_kernel}
\|\partial_\lambda^k\Imag R_{0,j}(\lambda^2+i0;r,s)\|_{L^{2,-\sigma}_r L^{2,-\sigma}_s} \le C\lambda^{n-2-k}
\end{equation}
for $0<\lambda \le 1.$ By \eqref{plamR0bd} we have
\begin{align*}
& \|\partial_\lambda^k\Imag R_{0,j}(\lambda^2+i0;r,s)\|_{L^{2,-\sigma}_r L^{2,-\sigma}_s}^2 \\
& \hspace{.5cm} =C\int\limits_0^\infty\int\limits_0^\infty \left[\partial_\lambda^k\lp J_{\nu_j}(\lambda r)J_{\nu_j}(\lambda s)\rp\right]^2 (1 + r)^{-2\sigma}(1 + s)^{-2\sigma}(rs)\,dr\,ds.
\end{align*}
Using the recursive formula for derivatives of the Bessel functions as before, we can reduce the proof to showing that 
\begin{equation}
\int\limits_0^\infty\int\limits_0^\infty r^{1+2\ell} s^{1+2m} |J_{\nu_j+\alpha}(\lambda r)|^2|J_{\nu_j+\beta}(\lambda s)|^2(1 + r)^{-2\sigma}(1 + s)^{-2\sigma}\,dr\,ds \le C\lambda^{2(n-2-k)}
\end{equation}
for any integers $\ell,\,m\ge 0$ with $\ell + m =k$ and integers $\,\alpha,\,\beta$ with $|\alpha|\le \ell$ and $|\beta|\le m$. Since the above integral is separable, it is in fact enough to show
\begin{equation}\label{r_imag_kernel}
\int\limits_0^\infty r^{1+2\ell} |J_{\nu_j+\alpha}(\lambda r)|^2(1 + r)^{-2\sigma}\,dr \le C\lambda^{n-2-2\alpha}
\end{equation}
for any $\ell \le k$ and $|\alpha|\le \ell,$ since analogous estimates will apply to the integral in the $s$ variable. First, notice that \eqref{small_arg} implies

\begin{align*}
& \int\limits_0^\frac{1}{\lambda} r^{1+2\ell} |J_{\nu_j+\alpha}(\lambda r)|^2(1 + r)^{-2\sigma}\,dr  \le C\int\limits_0^\frac{1}{\lambda} r^{1+2\ell}(\lambda r)^{2(\nu_j+\alpha)}(1 + r)^{-2\sigma}\,dr  \\
& \hspace{.5cm} = C\lambda^{2(\nu_j+\alpha)}\int\limits_0^\frac{1}{\lambda}r^{1 + 2(\ell + \alpha + \nu_j)}(1 + r)^{-2\sigma}\,dr  \le C'\lambda^{2(\nu_j+\alpha)}\int\limits_0^\frac{1}{\lambda}(1 + r)^{1 + 2(\ell +\alpha +\nu_j-\sigma)}\,dr\\
& \hspace{1cm} \le C''\lambda^{2(\nu_j+\alpha)}\left|\lp 1 + \frac{1}{\lambda}\rp^{2 + 2(\ell + \alpha +\nu_j - \sigma)} - 1\right| \\
&\hspace{1.2cm} \le C_1\lambda^{2(\sigma - \ell)-2}(\lambda + 1)^{2+ 2(\ell+\alpha+\nu_j-\sigma)} + C_2 \lambda^{2(\nu_j+\alpha)}\\
&\hspace{1.5cm} \le C_1\lambda^{2(\sigma - \ell) - 2} + C_2 \lambda^{2(\nu_j+\alpha)}.
\end{align*}
Recalling that $\sigma > \frac{n}{2} + k$ and $\ell \le k,$ we have that $2(\sigma - \ell) > n$. Also, we have $2(\nu_j+\alpha) \ge n-2 + 2\alpha$, and since $|\alpha|\le \ell\le k$, we have that the above is bounded by a constant times $\lambda^{n-2-2\alpha}$ for $0 < \lambda \le 1$ as claimed. 

Next, we consider the integral over the region where $\frac{1}{\lambda}\le r < \infty$. For this, we use \eqref{large_arg} to obtain
\begin{align*}
& \int\limits_\frac{1}{\lambda}^\infty r^{1+2\ell} |J_{\nu_j+\alpha}(\lambda r)|^2(1 + r)^{-2\sigma}\,dr \le C\int\limits_\frac{1}{\lambda}^\infty r^{1+2\ell}(\lambda r)^{-1}(1 + r)^{-2\sigma}\,dr\\
& \hspace{.5cm} \le \frac{C}{\lambda}\int\limits_{\frac{1}{\lambda}}^\infty (1 + r)^{2(\ell - \sigma)}\,dr 
\le \frac{C'}{\lambda}\lp 1 + \frac{1}{\lambda}\rp^{1 + 2(\ell-\sigma)}\\
& \hspace{1.5cm} = C'\lambda^{2(\sigma - \ell)-2}(\lambda + 1)^{1 + 2(\ell - \sigma)} \le C''\lambda^{2(\sigma - \ell) - 2}.
\end{align*}
The restrictions on $\sigma$ guarantee that the above is bounded by a constant times $\lambda^{n-2}$ for $0 < \lambda \le 1.$ Therefore, \eqref{r_imag_kernel} holds, and the proof is complete.
\end{proof}

Next, we aim to prove weighted $L^q$ estimates on the free radial resolvent kernels $R_{0,j}$, which enables us to control the terms in the Birman-Schwinger series for $R_{V,j}$ when applied iteratively. First, we make note of a technical lemma.


\begin{lemma}
\label{bessel_lemma}
Let $\nu \ge \frac{n-2}{2}$, and $\lambda > 0$. Suppose that $\beta,m\in\Z$ are such that $|\beta|\le m$ and $\nu + \beta \ge 0$. Assume also that $1\le q < \infty$ and that $\sigma > \frac{n}{q} + m$. Then there exist  $C_1,C_2 > 0$ such that 
\begin{equation}\label{J_bound}
\int\limits_0^\infty (\lambda s)^{q\lp m - \frac{n-2}{2}\rp}| J_{\nu+\beta}(\lambda s)|^q (1+s)^{-q\sigma} s^{n-1}\,ds \le \begin{cases}
C_1\lambda^{-n} + C_2\lambda^{q\lp m - \frac{n-1}{2}\rp}, & 1 \le \lambda < \infty\\
C\lambda^{q\sigma - n}, & 0 < \lambda \le 1.
\end{cases}
\end{equation}
\end{lemma}

\begin{proof}
Let us denote by $I(\lambda)$ the integral in the statement above, and observe that $I(\lambda)$ is clearly nonnegative for all $\lambda > 0$. If we split the integral into the the regions where $0 < s < \frac{1}{\lambda}$ and $\frac{1}{\lambda} < s < \infty,$ we can apply \eqref{small_arg} and \eqref{large_arg} to $J_{\nu+\beta}$ to obtain that
\[I(\lambda) \le C\int\limits_0^{\frac{1}{\lambda}}(\lambda s)^{q\lp\nu - \frac{n-2}{2} +\beta + m\rp}(1 + s)^{-q\sigma}s^{n-1}\,ds + C\int\limits_{\frac{1}{\lambda}}^\infty (\lambda s)^{q\lp m - \frac{n-1}{2}\rp}(1 + s)^{-q\sigma}s^{n-1}\,ds\]
for some constant $C>0$. To estimate these integrals, we treat the cases $\lambda \ge 1$ and $\lambda \le 1$ separately. First suppose that $\lambda \ge 1$. Then we see that
\begin{equation}
\int\limits_0^{\frac{1}{\lambda}}(\lambda s)^{q\lp\nu - \frac{n-2}{2} +\beta + m\rp}(1 + s)^{-q\sigma}s^{n-1}\,ds \le \lambda^{1-n}\int\limits_0^{\frac{1}{\lambda}}(1 + s)^{-q\sigma}\,ds \le C \lambda^{-n},
\end{equation}
since $\nu - \frac{n-2}{2} + \beta +m \ge 0$ and $q\sigma > 0$. For the integral over $\frac{1}{\lambda} < s < \infty$, we have
\begin{align*}
\int\limits_{\frac{1}{\lambda}}^\infty (\lambda s)^{q\lp m - \frac{n-1}{2}\rp}(1 + s)^{-q\sigma}s^{n-1}\,ds & = \lambda^{q\lp m - \frac{n-1}{2}\rp}\int\limits_{\frac{1}{\lambda}}^\infty s^{n-1 + q\lp m - \frac{n-1}{2}\rp}(1 + s)^{-q\sigma}\,ds.
\end{align*}
 Under the hypothesis that $\sigma > \frac{n}{q} + m,$ the integral
 \[\int\limits_{1}^\infty s^{n-1 + q\lp m - \frac{n-1}{2}\rp}(1 + s)^{-q\sigma}\,ds\]
 converges and is bounded by constant which is independent of $\lambda.$ For the region where $\frac{1}{\lambda} < s < 1$, we have 
 \[\int\limits_{\frac{1}{\lambda}}^1 s^{n-1+q\lp m - \frac{n-1}{2}\rp}(1 + s)^{-q\sigma}\,ds \le C\lambda^{-n - q\lp m - \frac{n-1}{2}\rp}.\]
 Thus, 
 \[\lambda^{q\lp m - \frac{n-1}{2}\rp}\int\limits_{\frac{1}{\lambda}}^\infty s^{q\lp m - \frac{n-1}{2}\rp}(1 + s)^{-q\sigma}s^{n-1}\,ds \le \max\{\lambda^{-n},\lambda^{q\lp m -\frac{n-1}{2}\rp}\}\]
 when $\lambda \ge 1.$
 
 Now take the case where $0 < \lambda \le 1.$ Then, since $|\beta| \le m$ and $\nu \ge \frac{n-2}{2}$, we have 
 \begin{align*}
 & \int\limits_0^{\frac{1}{\lambda}}(\lambda s)^{q\lp \nu - \frac{n-2}{2} +\beta + m\rp}(1 + s)^{-q\sigma}s^{n-1}\,ds  \le C\int\limits_0^{\frac{1}{\lambda}}(1 + s)^{n - 1 -q\sigma}\,ds\\
 & \hspace{1.5cm} = C' \lp 1 + \frac{1}{\lambda}\rp^{n - q\sigma}  \le C'' \lambda^{q\sigma - n}.
 \end{align*}
For the integral over $\frac{1}{\lambda} \le s < \infty$, we notice that 
\[
 \int\limits_{\frac{1}{\lambda}}^\infty (\lambda s)^{q\lp m - \frac{n-1}{2}\rp}(1 + s)^{-q\sigma}s^{n-1}\,ds \le \lambda^{q\lp m - \frac{n-1}{2}\rp}\int\limits_{\frac{1}{\lambda}}^\infty (1 + s)^{n-1 + q\lp m - \frac{n-1}{2} - \sigma\rp}\,ds
\]
since $1 \le \frac{1}{\lambda} \le s$. Recalling the assumption that $\sigma > \frac{n}{q}+ m$, we can see that 
\[ n - 1 + q\lp m - \frac{n-1}{2} - \sigma\rp < - 1,\]
and therefore,

\[
\begin{split}
& \lambda^{q\lp m - \frac{n-1}{2}\rp}\int\limits_{\frac{1}{\lambda}}^\infty (1 + s)^{n-1 + q\lp m - \frac{n-1}{2} - \sigma\rp}\,ds 
 \le C'\lambda^{q\lp m - \frac{n-1}{2}\rp}\lp 1  + \frac{1}{\lambda}\rp^{n+q\lp m - \frac{n-1}{2}-\sigma\rp} \\
 & \hspace{3cm} \le C''\lambda^{q\sigma - n}.
\end{split}
\]
 Therefore, $I(\lambda)\le C\lambda^{q\sigma - n}$ for $0< \lambda \le 1$. 
\end{proof}

Next, we establish some estimates on the $L^{q,\sigma}$ norms of $R_{0,j}(\lambda^2\pm i0)(r,s)$ when the norm is only taken with respect to one variable.

\begin{proposition}
\label{Lq_imaginary}
Let $k \ge 0$ be an integer. Also assume that $1 \le q < \infty$ and
\begin{equation}
\label{sigma_condition}
\sigma > \frac{n}{q} + k.
\end{equation}
Then for $\lambda \ge 1$, we have
\begin{equation}\label{imag_high_freq}
\begin{split}
& \|\partial_\lambda^k\Im R_{0,j}(\lambda^2+i0;r,\cdot)\|_{L^{q,\sigma}} \\ & \hspace{.5cm}\le C_{j,q,\sigma,k}\lambda^{n-2-k}\sum\limits_{\ell+m = k}\left[(1 + \lambda r)^{\ell - \frac{n-1}{2}}\lp C_1\lambda^{-\frac{n}{q}} + C_2\lambda^{m-\frac{n-1}{2}}\rp\right] 
\end{split}
\end{equation}
for some $C_{j,q,\sigma,k}> 0$. Furthermore, when $\lambda \le 1$, we have
\begin{equation}\label{imag_low_freq}
\|\partial_\lambda^k\Im R_{0,j}(\lambda^2+i0;r,\cdot)\|_{L^{q,\sigma}} \le C_{j,q,\sigma,k}\lambda^{n-2}(1 + \lambda r)^{k - \frac{n-1}{2}}.
\end{equation}
By symmetry, we also have the analogous estimates
\begin{align}
\begin{split}
& \|\partial_\lambda^k\Im R_{0,j}(\lambda^2+i0;\cdot,s)\|_{L^{q,\sigma}} \\
& \hspace{1.5cm} \le C_{j,q,\sigma,k}\lambda^{n-2-k}\sum\limits_{\ell+m = k}\left[(1 + \lambda s)^{\ell - \frac{n-1}{2}}\lp C_1\lambda^{-\frac{n}{q}} + C_2\lambda^{m-\frac{n-1}{2}}\rp\right] 
\end{split}
\end{align}
for $\lambda \ge 1$, and
\begin{equation}
\|\partial_\lambda^k\Im R_{0,j}(\lambda^2+i0;\cdot,s)\|_{L^{q,\sigma}} \le C_{j,q,\sigma,k}\lambda^{n-2}(1 + \lambda s)^{k - \frac{n-1}{2}}
\end{equation}
when $\lambda \le 1.$
\end{proposition}

\begin{rmk}
\label{optrmk}
\textnormal{
Note that in the special case where the order of differentiation is less than or equal to $\frac{n-1}{2}$, these estimates reduce to simple polynomial behavior in $\lambda$. However, if the number of derivatives exceeds this threshold value, we begin to see a non-uniformity with respect to the secondary radial variable. This phenomenon is why the spatial weights appear in the statement of \thmref{disp_est}.
}
\end{rmk}

\begin{proof}
 Recall that the kernel of $\Im R_{0,j}(\lambda^2+i0)$ has the explicit expression
\[\Im R_{0,j}(\lambda^2+i0;r,s) = \frac{\pi}{2}(rs)^{-\frac{n-2}{2}}J_{\nu_j}(\lambda r)J_{\nu_j}(\lambda s).\]
Since Bessel functions satisfy the recursion relation 
\[J_\nu'(x) = \frac{1}{2}\lp J_{\nu-1}(x) - J_{\nu+1}(x)\rp,\]
we see that $\partial_\lambda^k\Im R_{0,j}(\lambda^2+i0)$ can be written as a finite linear combination of terms of the form 
\begin{equation}\label{J_nu_alphabeta}
(rs)^{-\frac{n-2}{2}}r^\ell s^m J_{\nu_j+\alpha}(\lambda r)J_{\nu_j+\beta}(\lambda s),
\end{equation}
where $\ell,m,\alpha,\beta$ are integers satisfying $\ell + m = k$, $|\alpha|\le \ell,$ and $|\beta|\le m$. Therefore, by the triangle inequality, it suffices to estimate the weighted $L^q$ norms of such terms. Taking the $L^{q,\sigma}$ norm with respect to the $s$ variable in \eqref{J_nu_alphabeta} yields
\[
\lambda^{q(n-2 -k)}(\lambda r)^{q\lp \ell - \frac{n-2}{2}\rp}|J_{\nu_j + \alpha}(\lambda r)|^q \int\limits_0^\infty (\lambda s)^{q\lp m - \frac{n-2}{2}\rp}|J_{\nu_j + \beta}(\lambda s)|^q (1 + s)^{-q\sigma}s^{n-1}\,ds.
\]
Note that since $|\alpha|\le \ell$, we have that the product $(\lambda r)^{\ell - \frac{n-2}{2}}|J_{\nu_j+\alpha}(\lambda r)|$ is a continuous function of $\lambda r$, and thus by \eqref{large_arg} we obtain
\[(\lambda r)^{q\lp \ell - \frac{n-2}{2}\rp}|J_{\nu_j+\alpha}(\lambda r)|^q \le C(1 + \lambda r)^{q\lp\ell - \frac{n-1}{2}\rp}.\]
Thus, we have that the $L^{q,\sigma}$ norm of \eqref{J_nu_alphabeta} is bounded by 
\begin{equation}
\label{bessel_linear_comb}
 C \lambda^{q(n-2 -k)}(1 + \lambda r)^{q\lp \ell - \frac{n-1}{2}\rp} \int\limits_0^\infty (\lambda s)^{q\lp m - \frac{n-2}{2}\rp}|J_{\nu_j + \beta}(\lambda s)|^q (1 + s)^{-q\sigma}s^{n-1}\,ds.
\end{equation}
Now, observe that the integral above is in exactly the right form for us to apply \lemref{bessel_lemma}. Hence, we have that \eqref{bessel_linear_comb} is bounded by 
\[\begin{cases}
C\lambda^{q(n-2-k)}(1 + \lambda r)^{q\lp\ell - \frac{n-1}{2}\rp}\max\{\lambda^{-n},\lambda^{q\lp m - \frac{n-1}{2}\rp}\}, & \lambda \ge 1\\
C\lambda^{q(n-2-k) + q\sigma - n}(1 + \lambda r)^{q\lp\ell - \frac{n-1}{2}\rp}, & \lambda \le 1,
\end{cases}\]
for some possibly larger constant $C.$ In the case where $\lambda \ge 1$, simply taking $q$th roots gives estimate \eqref{imag_high_freq}. When $\lambda \le 1$, we can use that $\sigma$ satisfies \eqref{sigma_condition} to obtain that $q(n - 2 - k) + q\sigma - n >q(n - 2).$ Once again, taking $q$th roots gives \eqref{imag_low_freq}.

\end{proof}

Next, we estimate the $L^q$ norm of the resolvent when we do not take the imaginary part.

\begin{proposition}
\label{Lq_nonimaginary}
Let $k\ge 0$ be an integer and suppose $1 \le q\le \frac{n}{n-2}$. Then, if $\sigma$ satisfies \eqref{sigma_condition}, we have that when $\lambda \ge 1,$
\begin{equation}
\|\partial_\lambda^k R_{0,j}(\lambda^2+i0;r,\cdot)\|_{L^{q,\sigma}} \le C\lambda^{n-2-k}\sum\limits_{\ell+m = k}\left[(1 + \lambda r)^{\ell - \frac{n-1}{2}}\lp C_1 + C_2\lambda^{m-\frac{n-1}{2}}\rp\right] 
\end{equation}
 for some $C,C_1,C_2 > 0$. If $0 <\lambda \le 1$, then we have
\begin{equation}
\|\partial_\lambda^k R_{0,j}(\lambda^2+i0);r,\cdot)\|_{L^{q,\sigma}} \le C_1\lambda^{-k} + C_2\lambda^{n-2-k}(1 + \lambda r)^{k - \frac{n-1}{2}}.
\end{equation} 
Under the same assumptions on $\sigma,$ we also have
\begin{equation}
\|\partial_\lambda^k R_{0,j}(\lambda^2\pm i0;\cdot,s)\|_{L^{q,\sigma}} \le C\lambda^{n-2-k}\sum\limits_{\ell+m = k}\left[(1 + \lambda s)^{\ell - \frac{n-1}{2}}\lp C_1 + C_2\lambda^{m-\frac{n-1}{2}}\rp \right] 
\end{equation}
when $\lambda \ge 1$, and 
\begin{equation}
\|\partial_\lambda^k R_{0,j}(\lambda^2+i0;\cdot,s)\|_{L^{q,\sigma}} \le  C_1\lambda^{-k} + C_2\lambda^{n-2-k}(1 + \lambda s)^{k - \frac{n-1}{2}}
\end{equation}
when $0 <\lambda \le 1.$ 
\end{proposition}

\begin{proof}
Recalling that 
\[R_{0,j}(\lambda^2+i0)(r,s) = \begin{cases}
\frac{\pi i}{2}J_{\nu_j}(\lambda s)\Hone_{\nu_j}(\lambda r), & s < r\\
\frac{\pi i}{2}J_{\nu_j}(\lambda r)\Hone_{\nu_j}(\lambda s), & s > r,
\end{cases}\]
and \eqref{Bessel_derivative}, we see that when $s < r$, $\partial_\lambda^k R_{0,j}(\lambda^2+i0;r,s)$ can be written as a finite linear combination of terms of the form 
\[(rs)^{-\frac{n-2}{2}} r^\ell s^m J_{\nu_j +\alpha}(\lambda s)\Hone_{\nu_j + \beta}(\lambda r),\]
 where, as in the proof of \propref{Lq_imaginary}, $\ell,m$ are nonnegative integers with $\ell + m = k$ and $\alpha,\beta$ are any integers with $|\alpha|\le \ell$ and $|\beta|\le m.$ Similarly, when $r < s$, we can write $\partial_\lambda^k R_{0,j}(\lambda^2+i0)(r,s)$ as a combination of terms of the same form, but with the roles of $r$ and $s$ reversed. Therefore, it suffices to estimate 
 \begin{equation}\label{s<r}
 I(\lambda ,r) : =\|(rs)^{-\frac{n-2}{2}}r^\ell s^m J_{\nu_j+ \beta}(\lambda s)\Hone_{\nu_j+\alpha}(\lambda r)\rho^{-\sigma}(s)\mathds 1_{\{s < r\}}\|_{L_s^q}^q 
 \end{equation}
 and 
 \begin{equation}\label{s>r}
\II(\lambda,r) : =\|(rs)^{-\frac{n-2}{2}}r^\ell s^m J_{\nu_j + \alpha}(\lambda r)\Hone_{\nu_j+\beta}(\lambda s)\rho^{-\sigma}(s)\mathds 1_{\{s > r\}}\|_{L^q_s}^q
 \end{equation}
 for any $\ell,m,\alpha,\beta$ as above. 


We first estimate $I(\lambda,r)$ in the case where $\lambda r \ge 1.$ Under this hypothesis, we can apply \eqref{large_arg} to obtain
\[
I(\lambda,r) \le C\lambda^{q(n-2 -k)}(\lambda r)^{q\lp \ell-\frac{n-1}{2}\rp}\int\limits_0^r (\lambda s)^{q\lp m-\frac{n-2}{2}\rp}|J_{\nu_j + \beta}(\lambda s)|^q(1 + s)^{-q\sigma}s^{n-1}\,ds.
\]
We now apply \lemref{bessel_lemma} to the integral above, which gives
\begin{equation}\label{I_eqn_1}
I(\lambda, r) \le \begin{cases}
C\lambda^{q(n-2-k)}(\lambda r)^{q\lp\ell - \frac{n-1}{2}\rp}\max\{\lambda^{-n},\lambda^{q\lp m - \frac{n-1}{2}\rp}\}, &\lambda r \ge 1,\,\lambda \ge 1\\
C\lambda^{q(n-2 -k + \sigma) - n}(\lambda r)^{q\lp\ell - \frac{n-1}{2}\rp},& \lambda r \ge 1,\,  0<\lambda \le 1.
\end{cases}
\end{equation}

Now let us consider the case where $\lambda r \le 1$. Here we can apply \eqref{small_arg}, which gives
\begin{equation}\label{I_eqn2}
I(\lambda,r) \le C\lambda^{q(n-2-k)}(\lambda r)^{q\lp \ell-\frac{n-2}{2} - \nu_j-\alpha\rp}\int\limits_0^r (\lambda s)^{q\lp m -\frac{n-2}{2} +\nu_j +\beta\rp}(1 + s)^{-q\sigma}s^{n-1}\,ds.
\end{equation}
If $r \le 1,$ we can bound the right-hand side of \eqref{I_eqn2} by 
\begin{align*}
\begin{split}
& C\lambda^{q(n-2- k)}(\lambda r)^{q\lp k - \alpha + \beta - (n-2)\rp}r^{n-1}\int\limits_0^r(1 + s)^{-q\sigma}\,ds \le \wt C\lambda^{q(n-2-k)}(\lambda r)^{ -q(n-2)} r^n \\
& \hspace{2cm} = \wt C\lambda^{-qk}r^{n-q(n-2)},
\end{split}
 \end{align*}
since $k - \alpha + \beta \ge 0$ and $\int\limits_0^r(1 + s)^{-q\sigma}\,ds \le C'r $ for some $C' > 0.$ Recalling that $q \le \frac{n}{n-2}$, we obtain 
\begin{equation}\label{I_eqn3}
I(\lambda,r) \le C\lambda^{-qk}, \hskip 0.2in \lambda r \le 1,\, r\le 1.
\end{equation}
Now, if $r \ge 1$, we can bound the right-hand side \eqref{I_eqn2} by
\begin{equation}\label{I_eqn4}
C\lambda^{q(n-2-k)}(\lambda r)^{q\lp k-\alpha+\beta - (n-2)\rp}\int\limits_0^r (1 + s)^{-q\sigma}s^{n-1}\,ds \le C\lambda^{q(n-2-k)}(\lambda r)^{-q(n-2)}(1 + r)^{n-q\sigma}
\end{equation}
since $k-\alpha + \beta \ge 0$ and $\lambda r \le 1.$ Recalling that $\sigma > \frac{n}{q}$, we have that the right-hand side of \eqref{I_eqn3} by
\begin{equation}\label{I_eqn5}
C\lambda^{-qk}r^{-q(n-2)} \le C\lambda^{-qk}.
\end{equation}
Combining \eqref{I_eqn3} and \eqref{I_eqn5}, we have that 
\begin{equation}\label{I_eqn6}
I(\lambda,r)\le C\lambda^{-qk},\quad \lambda r \le 1.
\end{equation}
Combining \eqref{I_eqn6} with \eqref{I_eqn_1}, we have that 
\begin{equation}\label{I_final_estimate}
I(\lambda,r) \le \begin{cases}C\lambda^{q(n-2-k)}(1 + \lambda r)^{q\lp\ell - \frac{n-1}{2}\rp}\max\{\lambda^{-n},\lambda^{q\lp m - \frac{n-1}{2}\rp}\}, & \lambda \ge 1\\
C_1\lambda^{-qk} + C_2\lambda^{q(n-2)}(1 + \lambda r)^{q\lp\ell-\frac{n-1}{2}\rp}, & 0 < \lambda \le 1.
\end{cases}
\end{equation}


Next, we move on to estimating $\II(\lambda ,r)$. Again we consider the cases $\lambda r \ge 1$ and $\lambda r \le 1$ separately. For $\lambda r \ge 1,$ we apply \eqref{small_arg} and \eqref{large_arg} to obtain
\begin{align*}
\II(\lambda,r) &\le C\lambda^{q(n-2-k)}(\lambda r)^{q\lp\ell - \frac{n-1}{2}\rp}\int\limits_{\frac{1}{\lambda}}^\infty (\lambda s)^{q\lp m - \frac{n-1}{2}\rp}(1 + s)^{-q\sigma}s^{n-1}\,ds.\\
\end{align*}
We can then repeat arguments from the proof of \lemref{bessel_lemma} to obtain
\begin{equation}\label{II_eqn_1}
\II(\lambda,r) \le\begin{cases} C\lambda^{q(n-2-k)}(\lambda r)^{q\lp \ell - \frac{n-1}{2}\rp}\max\{\lambda^{-n},\lambda^{q\lp m - \frac{n-1}{2}\rp}\}, & \lambda r\ge 1 ,\,\lambda \ge 1\\
C\lambda^{q(n-2-k)}(\lambda r)^{q\lp\ell - \frac{n-1}{2}\rp}, & \lambda r\ge 1, \,\lambda \le 1.
\end{cases}
\end{equation}

Now consider the case where $\lambda r \le 1$. Here we rewrite $\II(\lambda ,r)$ as
\[
\lambda^{q(n-2-k)}(\lambda r)^{q\lp\ell-\frac{n-2}{2}\rp}|J_{\nu_j+\alpha}(\lambda r)|^q\lp\int\limits_r^{\frac{1}{\lambda}}+\int\limits_{\frac{1}{\lambda}}^\infty\rp(\lambda s)^{q\lp m -\frac{n-2}{2}\rp}|\Hone_{\nu_j + \beta}(\lambda s)|^q(1+s)^{-q\sigma}s^{n-1}\,ds.
\]
For the integral over $\frac{1}{\lambda} < s < \infty$, we can apply \eqref{large_arg} to $\Hone_{\nu_j+\beta}$ and \eqref{small_arg} to $J_{\nu_j+\alpha}$ and repeat previous calculations to show that
\begin{align}\label{II_eqn_2}
\begin{split}
\lambda^{q(n-2-k)}(\lambda r)^{q\lp\ell-\frac{n-2}{2}\rp}|J_{\nu_j+\alpha}(\lambda r)|^q\int\limits_{\frac{1}{\lambda}}^\infty(\lambda s)^{q\lp m -\frac{n-2}{2}\rp}|\Hone_{\nu_j + \beta}(\lambda s)|^q(1+s)^{-q\sigma}s^{n-1}\,ds\\
& \hspace{-4in}\le\begin{cases}
C\lambda^{q(n-2-k)}\max\{\lambda^{-n},\lambda^{q\lp m -\frac{n-1}{2}\rp}\}, & \lambda r \le 1, \,\lambda \ge 1\\
C\lambda^{q(n-2-k)}\lambda^{q\sigma - n}, & \lambda r\le 1,\,\lambda \le 1
\end{cases}\\
& \hspace{-4in}\le \begin{cases}
C\lambda^{q(n-2-k)}\max\{\lambda^{-n},\lambda^{q\lp m -\frac{n-1}{2}\rp}\}, & \lambda r \le 1, \,\lambda \ge 1\\
C\lambda^{q(n-2)}, & \lambda r\le 1,\,\lambda \le 1,
\end{cases}
\end{split}
\end{align}
where the last inequality follows since $\sigma > \frac{n}{q}+k.$ Now, in the region where $r < s < \frac{1}{\lambda}$, we must apply \eqref{small_arg}, which yields
\begin{align*}
\int\limits_r^{\frac{1}{\lambda}}(\lambda s)^{q\lp m - \frac{n-2}{2}\rp}|\Hone_{\nu_j+\beta}(\lambda s)|^q(1+s)^{-q\sigma}s^{n-1}\,ds &\le C \int\limits_r^{\frac{1}{\lambda}}(\lambda s)^{q\lp m - \frac{n-2}{2} - \nu_j -\beta\rp}(1+s)^{-q\sigma}s^{n-1}\,ds\\
&\hspace{-0.9in} = C\lambda^{q\lp m - \frac{n-2}{2}-\nu_j - \beta\rp}\int\limits_r^{\frac{1}{\lambda}}s^{n - 1 +q\lp m - \frac{n-2}{2} - \nu_j -\beta\rp}(1 + s)^{-q\sigma}\,ds.
\end{align*}
If $\lambda \ge 1,$ then $(1 + s)^{-q\sigma}$ is bounded by a uniform constant for all $r < s < \frac{1}{\lambda}$, and so the above is bounded by 
\[C\lp \lambda^{-n} - r^n(\lambda r)^{q\lp m - \frac{n-2}{2} -\nu_j-\beta\rp}\rp\]
after possibly increasing $C$. We note that under our assumptions on $r$ and $\lambda,$ this quantity is still nonnegative. Combining this with \eqref{small_arg} applied to $J_{\nu_j+\alpha},$ we obtain
\begin{align}
\begin{split}
\lambda^{q(n-2-k)}(\lambda r)^{q\lp\ell - \frac{n-2}{2}\rp}|J_{\nu_j+\alpha}(\lambda r)|^q\int\limits_r^{\frac{1}{\lambda}}(\lambda s)^{q\lp m - \frac{n-2}{2}\rp}|\Hone_{\nu_j+\beta}(\lambda s)|^q(1+s)^{-q\sigma}s^{n-1}\,ds\\
& \hspace{-4in}\le C\lambda^{q(n-2-k)}(\lambda r)^{q\lp\ell - \frac{n-2}{2} + \nu_j + \alpha\rp}\lp\lambda^{-n} - r^n(\lambda r)^{q\lp m - \frac{n-2}{2} - \nu_j - \beta\rp}\rp\\
& \hspace{-4in} \le \lambda^{q(n-2-k)}\left[C_1\lambda^{-n}(\lambda r)^{q\lp\ell - \frac{n-2}{2} + \nu_j + \alpha\rp}  +C_2 r^{n - q(n-2)}(\lambda r)^{q(k + \alpha - \beta)}\right],
\end{split}
\end{align}
for some $C_1,C_2 > 0.$ Recalling that $|\alpha|\le \ell$, $\nu_j \ge \frac{n-2}{2}$, and $|\alpha| + |\beta|\le k$, we obtain
\begin{align}\label{II_eqn_3}
\begin{split}
\lambda^{q(n-2-k)}(\lambda r)^{q\lp\ell - \frac{n-2}{2}\rp}|J_{\nu_j+\alpha}(\lambda r)|^q\int\limits_r^{\frac{1}{\lambda}}(\lambda s)^{q\lp m - \frac{n-2}{2}\rp}|\Hone_{\nu_j+\beta}(\lambda s)|^q(1+s)^{-q\sigma}s^{n-1}\,ds\\
& \hspace{-3in} \le \lambda^{q(n-2-k)}\left[ C_1 \lambda^{-n} + C_2 r^{n - q(n-2)}\right]\\
& \hspace{-3in}\le C \lambda^{q(n-2-k)}
\end{split}
\end{align}
for $\lambda r \le 1$ and $\lambda \ge 1,$ since $q \le \frac{n}{n-2}$.

Finally, we consider the same integral over $r < s < \frac{1}{\lambda}$, once again where $0<\lambda r \le 1$ but with $\lambda \le 1.$ For this, we further subdivide into the cases where $r \le 1$ and $r\ge 1$. If $r \le 1$, we split the integral into the regions where $r < s < 1$ and $1 < s < \frac{1}{\lambda}$. For the integral over $r < s < 1$, we can repeat the above argument to obtain the same bound as in \eqref{II_eqn_3}. To bound the integral over $1 < s < \frac{1}{\lambda}$, we use \eqref{small_arg} to obtain
\begin{align*}
 \lambda^{q(n-2-k)}(\lambda r)^{q\lp \ell - \frac{n-2}{2}\rp}|J_{\nu+\alpha}(\lambda r)|^q\int\limits_1^{\frac{1}{\lambda}}(\lambda s)^{q\lp m - \frac{n-2}{2}\rp}|\Hone_{\nu_j+\beta}(\lambda s)|^q(1 + s)^{-q\sigma}s^{n-1}\,ds & \\
& \hspace{-4.5in}\le  \lambda^{q(n-2-k)}(\lambda r)^{q\lp \ell - \frac{n-2}{2} + \nu_j + \alpha\rp}\int\limits_1^{\frac{1}{\lambda}}(\lambda s)^{q\lp m  -\frac{n-2}{2} - \nu_j -\beta\rp}(1 + s)^{-q\sigma}s^{n-1}\,ds\\
& \hspace{-4.5in}= \lambda^{q(\alpha - \beta)}r^{q\lp\ell - \frac{n-2}{2}+ \nu_j + \alpha\rp}\int\limits_1^\frac{1}{\lambda}s^{q\lp m - \frac{n-2}{2} -\nu_j-\beta\rp + n-1}(1 + s)^{-q\sigma}\,ds\\
& \hspace{-4.5in}\le \lambda^{-qk}\int\limits_1^\infty s^{qk+n-1}(1 + s)^{-q\sigma}\,ds\\
& \hspace{-4.5in}\le C\lambda^{-qk}, 
\end{align*}
where the last inequality follows from the fact that $\sigma > \frac{n}{q} + k.$ Now, if $r \ge 1$, we have
\begin{align*}
 \lambda^{q(n-2-k)}(\lambda r)^{q\lp \ell - \frac{n-2}{2}\rp}|J_{\nu+\alpha}(\lambda r)|^q\int\limits_r^{\frac{1}{\lambda}}(\lambda s)^{q\lp m - \frac{n-2}{2}\rp}|\Hone_{\nu_j+\beta}(\lambda s)|^q(1 + s)^{-q\sigma}s^{n-1}\,ds & \\
& \hspace{-4.5in}\le  \lambda^{q(n-2-k)}(\lambda r)^{q\lp \ell - \frac{n-2}{2} + \nu_j + \alpha\rp}\int\limits_r^{\frac{1}{\lambda}}(\lambda s)^{q\lp m  -\frac{n-2}{2} - \nu_j -\beta\rp}(1 + s)^{-q\sigma}s^{n-1}\,ds\\\\
& \hspace{-4.5in} \le C\lambda^{q(\alpha - \beta)}r^{q\lp\ell - \frac{n-2}{2} +\nu_j+\alpha\rp}\int\limits_r^\infty s^{q\lp m - \frac{n-2}{2}-\nu_j-\beta\rp-q\sigma +n-1}\,ds\\
& \hspace{-4.5in}\le C\lambda^{q(\alpha - \beta)}r^{q\lp k +\alpha - \beta - (n-2)\rp - q\sigma +n}\\
& \hspace{-4.5in}\le C\lambda^{q(\alpha - \beta)}r^{q\lp\alpha-\beta - (n-2)\rp},
\end{align*}
where in the last inequality we once again used that $\sigma > \frac{n}{q} + k.$ Now, since $r \le \frac{1}{\lambda}$, we have that the above is bounded by a constant times $\lambda^{q(n-2)}$ for $\lambda \le 1.$ Combining this with \eqref{II_eqn_1}, \eqref{II_eqn_2}, and \eqref{II_eqn_3}, we obtain
\begin{equation}\label{II_final_estimate}
\II(\lambda,r) \le \begin{cases}C\lambda^{q(n-2-k)}(1 + \lambda r)^{q\lp\ell - \frac{n-1}{2}\rp}\max\{1,\lambda^{q\lp m-\frac{n-1}{2} \rp}\}, & \lambda \ge 1\\
C_1\lambda^{-qk} + C_2\lambda^{q\lp n-2-k\rp}(1 + \lambda r)^{q\lp\ell - \frac{n-1}{2}\rp}, & 0<\lambda \le 1.
\end{cases}
\end{equation}
In light of \eqref{I_final_estimate} and \eqref{II_final_estimate}, taking $q$th roots completes the proof of \propref{Lq_nonimaginary}.
\end{proof}

\section{Operator estimates for $R_V$}
\label{spectheory}

In this section we establish some weighted operator norm estimates for the perturbed radial resolvents $R_{V,j}$, defined via the mode-by-mode decomposition of $R_V(z^2)$:
\[R_V(z^2) = \sum\limits_{j=0}^\infty R_{V,j}(z^2)E_j.\]
Since $V$ is radial, it follows that we can write 
\begin{equation}\label{RV_j}
R_{V,j}(z^2) = \lp\partial_r^2 + \frac{n-1}{r}\partial_r +z^2 - \frac{\mu_j^2}{r^2} + V(r)\rp^{-1},
\end{equation}
wherever this inverse is well defined. Here, we prove that the mapping properties established for $R_{0,j}$ in \propref{L2_L2_prop} and \propref{L2_L2_imaginary} extend to $R_{V,j}$. Similar weighted estimates for Schr\"odinger operators on hyperbolic space are given in Section $4$ of \cite{borthwick2015dispersive}, and the techniques therein follow an analogous structure.

For a potential $V \in  \rho^{-2\sigma} L^\infty ( \mathbb{R}^+)$ with $\sigma > \frac{1}{2}$, the operator norm $\norm{VR_0(z^2)}_{L^2\to L^2}$ is small for
$\Im z$ large by the standard resolvent norm estimate on $R_0(z^2)$, which is computable in a similar fashion to that discussed in \propref{L2_L2_prop}.  
Hence, the operator $1 + VR_0(z^2)$ is invertible by Neumann series for large $\Im z$.  For $z$ in this range, we can write
\[
\begin{split}
R_{V,j}(z^2)= R_{0,j}(z^2)(1 + VR_{0,j}(z^2))^{-1}.
\end{split}
\]

\noindent We begin our analysis of these perturbed resolvents by proving \thmref{LAP}, which we recall states that $R_{0,j}$ admits a meromorphic continuation to the logarithmic cover of $\C\setminus\{0\}$ and satisfies the limiting absorption principle.

\begin{proof}[Proof of \thmref{LAP}]
As mentioned in \secref{free}, the meromorphic extension of
\[
\chi R_{0,j} \chi \]
with $\chi$ a smooth, compactly supported function, follows from \cite{baskin2019scattering}. The meromorphic continuation of $\chi R_{V,j} (\lambda) \chi$ follows from the
 work of Guillop{\'e}--Zworski~\cite{guillope1995upper} and the compactness of the resolvent on a compact manifold with a conic singularity, which can be seen for instance in the treatment of domains for conic operators in the work of Melrose-Wunsch \cite{melrose2004propagation}.  Using the techniques from Section \ref{R0_estimates}, we can easily see that $ \rho(r)^{- \eta} R_{0,j}(z)$ is compact as an operator on $  L^{2,\delta}$ provided that $\Imag (z^2) > 0$ and $\eta > \delta$, since the upper bounds \eqref{small_arg} and \eqref{large_arg} remain valid for arguments in the upper half plane.   



Next, we prove the limiting absorption principle for $R_{V,j}$. Following \cite{goldberg2004dispersive}, we observe that mapping properties of $R_{V,j}$ can be deduced from the estimates established for $R_{0,j}$. By the resolvent identity
\[
R_{0,j}(z^2) = R_{V,j}(z^2) + R_{V,j}(z^2)VR_{0,j}(z^2),
\]
we can write 
\[
R_{0,j}(z^2)\rho^{-\sigma} = R_{V,j}(z^2)\rho^{-\sigma}(1 + \rho^{\sigma}VR_{0,j}(z^2)\rho^{-\sigma})
\]
for $\rho (r) =1 + r$.  
The factor on the right is meromorphically invertible by the analytic Fredholm theorem, so that
\begin{equation}\label{Rvp}
R_{V,j}(z^2)\rho^{-\sigma} = R_{0,j}(z^2)\rho^{-\sigma} (1 + \rho^{\sigma}VR_{0,j}(z^2)\rho^{-\sigma})^{-1},
\end{equation}
whenever the inverse exists. By Proposition  \ref{L2_L2_prop} and the fact that $\rho^{\sigma}V = \rho^{2\sigma}V\rho^{-\sigma}$, we have 
\[
\norm{\rho^{\sigma}VR_{0,j}(\lambda^2\pm i0)\rho^{-\sigma}}_{L^2\to L^2} \le C \norm{\rho^{2\sigma} V}_{L^\infty} \abs{\lambda}^{-1}.
\]
Hence for $V \in \rho^{-2\sigma} L^\infty$ , there exists a constant $M_{V}$ such that for 
$\abs{\lambda} \ge M_{V}$,
\[
\norm{\rho^{\sigma}VR_{0,j}(\lambda^2\pm i0)\rho^{-\sigma}}_{L^2\to L^2}\le \frac12,
\]
implying that $(1 + \rho^{\sigma}VR_0(\lambda^2+i0)\rho^{-\sigma})^{-1}$ exists and satisfies
\[
\norm{(1 + \rho^{\sigma}VR_{0,j}(\lambda^2\pm i0)\rho^{-\sigma})^{-1}}_{L^2\to L^2} \le 2.  
\]
The estimates then follow from \eqref{Rvp} and Proposition  \ref{L2_L2_prop}. 
\end{proof}

As stated in the introduction, we postpone the proof of \thmref{absence.thm} until the appendices, and so we take the absence of embedded eigenvalues and resonances in the range $(0,\infty)$ as given for now. With this property in hand, our next goal is to prove the low energy estimates estimates from \thmref{low_energy_thm} under the assumption that $-\Delta_{C(X)} + V$ does not have a resonance at zero energy. 

\begin{proof}[Proof of \thmref{low_energy_thm}]
We prove the desired estimates only for $R_{V,j}(\lambda^2+i0)$, since the proof is analogous for $R_{V,j}(\lambda^2 - i0).$ Using a resolvent expansion motivated by \cite{borthwick2015dispersive}, we observe that
\[
R_{V,j}(\lambda+ i0) = R_{0,j}(\lambda+ i0) [ I + V R_{0,j} (\lambda+ i0)]^{-1}.
\]
Hence, if we can establish boundedness and regularity of $[ I + V R_{0,j}(\lambda^2+i0)]^{-1}$ through $\lambda = 0$, then \eqref{eqn:RV_L2bd} follows immediately from \propref{L2_L2_prop}. We observe that boundedness and regularity of the operator $(I + R_{0,j} (\lambda^2+i0) V)^{-1}$ follows from \thmref{absence.thm} and the assumption that $0$ is not a resonance or an eigenvalue of $-\Delta_{C(X)} + V$ and hence the boundedness of 
\[
(I + VR_{0,j}(\lambda^2 - i0))^{-1} = [(I + R_{0,j}(\lambda^2+i0) V)^{-1}]^*
\]
follows from analytic Fredholm theory.  Thus, we may extend \eqref{LAP_eqn} through $\lambda = 0$ to arrive at \eqref{eqn:RV_L2bd}.

To estimate $\partial_\lambda^k\Imag R_{V,j}$, we first consider the case where $k = 0$ and establish the pointwise bounds in $\lambda$. For this, we take note of the following resolvent identity
\begin{align}
\begin{split}
\label{resid}
&R_{V,j} (\lambda^2+i0 ) - R_{V,j} (\lambda^2 - i0)\\
& \hspace{1cm}= (I + R_{0,j} (\lambda^2+i0 )V)^{-1} [R_{0,j}(\lambda^2+i0) - R_{0,j} (\lambda^2 - i0)] (I + VR_{0,j} (\lambda^2 - i0))^{-1}.
\end{split}
\end{align}
This shows that the behavior of $\Imag R_{V,j}(\lambda^2+i0)$ near $\lambda = 0$ is the same as that of $\Im R_{0,j}(\lambda^2+i0)$, provided that the operators
\[
(I + R_{0,j} (\lambda^2+i0)V)^{-1} \ \ \text{and} \ \ (I + VR_{0,j}(\lambda^2 - i0))^{-1}
\]
are bounded for $\lambda$ in a neighborhood of $0$, which we have already observed earlier in the proof.  As a result, the $k=0$ bound in \eqref{RV_imaginarypart} clearly follows. The results for $k > 0$ then follow by differentiating term by term and applying \propref{L2_L2_imaginary}.
\end{proof}

\section{Full spectral resolution estimates}
\label{spectralresolution}

With the mapping properties for both the free and perturbed resolvents established in the previous sections, we are now able to obtain some precise pointwise estimates on the Schwartz kernel of $\Im R_{V,j}(\lambda^2\pm i0;r,s)$. 
\begin{proposition}\label{RV_pointwise}
Let $k \geq 0$ be an integer.  Suppose $V\in \rho^{-2\sigma}L^\infty(\R^+)$ with
 \begin{equation}\label{RV_sigmacondition}\sigma > 4\left\lceil\frac{n}{4}\right\rceil - 2+k,
 \end{equation}
 then for any $\alpha \ge \max\{k - \frac{n-1}{2},0\}$, 
\begin{equation}\label{RV_highfreq_weighted}
\sup\limits_{r,s>0}\left| \rho^{-\alpha}(r)\partial_\lambda^k\Imag R_{V,j}(\lambda^2\pm i0;r,s)\rho^{-\alpha}(s)\right| \le C_{j,k,V}\lambda^{2\lceil\frac{n}{4}\rceil(n-2)-1} 
\end{equation}
for all $\lambda\ge 1$ and some $C>0$. Furthermore, if $0 < \lambda \le 1,$ we have that
\begin{equation}\label{RV_lowfreq_weighted}
\sup\limits_{r,s>0}\left| \rho^{-\alpha}(r)\partial_\lambda^k\Imag R_{V,j}(\lambda^2\pm i0;r,s)\rho^{-\alpha}(s)\right| \le C_{j,k,V}\lambda^{n-2-k},
\end{equation}
under the same restrictions on $\alpha.$ 
\end{proposition}

The proof proceeds similarly to \cite[\S6]{borthwick2015dispersive}, which utilizes the following modified version of Young's inequality.
\begin{lemma}\label{young}
Suppose that on a measure space $(Y,\mu)$ the integral kernels $K_j(z,w)$, $j=1,2,$ satisfy
\[\|K_1(z,\cdot)\|_{L^{q_1}} \le A,\hskip 0.2in \|K_1(\cdot,w)\|_{L^{q_1}}\le A, \hskip 0.2in \|K_2(\cdot,w')\|_{L^{q_2}} \le B\]
uniformly in $z,w,w'$ for $q_1,q_2\in [1,\infty]$. Then if $\frac{1}{q_1} + \frac{1}{q_2} = \frac{1}{p} + 1$, we have that
\[\left\|\int K_1(\cdot,w)K_2(w,w')\,d\mu(w)\right\|_{L^p} \le AB\]
uniformly in $w'.$ The bound on $\|K_1(\cdot,w)\|_{L^{q_1}}$ is not required if $p = \infty.$
\end{lemma}

\noindent With this lemma in hand, we proceed to the proof of \propref{RV_pointwise}. 
\begin{proof}[Proof of \propref{RV_pointwise}]
We begin by expanding $R_{V,j}$ in a Birman-Schwinger series at all frequencies, as in \cite{goldberg2004dispersive}, which gives
\begin{equation}\label{Birman-Schwinger}
R_{V,j}(\tau) = \sum\limits_{\ell = 0}^{2M - 1}R_{0,j}(\tau)(-VR_{0,j}(\tau))^\ell + [R_{0,j}(\tau)V]^M R_{V,j}(\tau)[V R_{0,j}(\tau)]^M.
\end{equation}

\noindent As previously discussed, it suffices to consider only the case where we choose $\lambda^2+i0$ with $\lambda > 0.$ For simplicity, we write $R_{0,j}$ for $R_{0,j}(\lambda^2+i0)$ and $R_{V,j}$ for $R_{V,j}(\lambda^2+i0)$. We first consider the remainder term $[R_{0,j}V]^MR_{V,j}[V R_{0,j}]^M$. Since $V\in \rho^{-2\sigma}L^\infty$, we may write $V(r) = \rho^{-2\sigma}(r)f(r)$ for some $f\in L^\infty(\R^+).$ Also, note that for any two operators with Schwartz kernels $A(r,s),\,B(r,s)$, the kernel of their composition is given by 
\[\la A(r,\cdot),\overline{B(\cdot,s)}\ra_{L^2(\R^+)} = \langle B(\cdot,s),\overline{A(r,\cdot)}\rangle_{L^2(\R^+)},\] provided the composition makes sense. Therefore, we can write
\begin{align}\label{remainder_pairing}
\begin{split}
\rho^{-\alpha}[R_{0,j}V]^M R_{V,j}[V R_{0,j}]^M\rho^{-\alpha}(r,s) = \left\langle (\rho^{-\sigma}R_{V,j}\rho^{-\sigma})A(\cdot,s),A^*(r,\cdot)\right\rangle_{L^2},
\end{split}
\end{align}
where 
\[A(r,s) = (\rho^{-\sigma}fR_{0,j}\rho^{-\sigma})^{M-1}(\rho^{-\sigma}f R_{0,j}\rho^{-\alpha})(r,s),\]
and $A^*$ denotes the adjoint with respect to the $L^2$ pairing. However, we know that $A^*(r,s) = \overline{A(s,r)}$, and hence we can express the right-hand side of \eqref{remainder_pairing} as 
\[\la (\rho^{-\sigma}R_{V,j}\rho^{-\sigma})A(\cdot,s),\overline{A(\cdot,r)}\ra_{L^2}.\]
By \eqref{eqn:RV_L2bd}, we have that 
\begin{equation}\label{pairing_est}
\left|\partial_\lambda^k\la(\rho^{-\sigma}R_{V,j}\rho^{-\sigma})A(\cdot,s), \overline{A(\cdot,r)}\ra_{L^2}\right| \le \frac{C}{\la\lambda\ra}\max\limits_{k_1+k_2\le k}\lp\ltwo{\partial_\lambda^{k_1}A(\cdot,s)}\ltwo{\partial_\lambda^{k_2}A(\cdot,r)}\rp.
\end{equation}
To estimate the norms on the right, we wish to iteratively apply \lemref{young} to each factor in the definition of $A$. For this we consider the high and low frequency cases separately. First, suppose $\lambda \ge 1$. By \propref{Lq_nonimaginary}, we have that for $1\le q\le \frac{n}{n-2}$ and any $0 \le \wt k\le k_1,$
\begin{equation}\label{Lq_est}
\lpn{\rho^{-\sigma}(r)\partial_\lambda^{\wt k}R_{0,j}\rho^{-\sigma}(r,\cdot)}{q} \le  C\rho^{-\sigma}(r) \lambda ^{n-2-\wt k}\sum\limits_{\ell+m = \wt k}\left[(1 + \lambda r)^{\ell - \frac{n-1}{2}}\lp C_1 + C_2\lambda^{m-\frac{n-1}{2}}\rp\right].
\end{equation}
Note that if $\ell\le \frac{n-1}{2}$, we can see that the corresponding term in \eqref{Lq_est} is bounded by a constant times 
\[ \lambda ^{n-2-\wt k}\max\{1,\lambda^{(\wt k-\ell) - \frac{n-1}{2}}\} \le \max\{\lambda^{n-2-\wt k},\lambda^{\frac{n-3}{2}-\ell}\}\]
 uniformly for $r\in[0,\infty)$. On the other hand, if $\ell > \frac{n-1}{2}$ then we have that
\begin{equation}\label{nonuniform}
\rho^{-\sigma}(r)(1 + \lambda r)^{\ell-\frac{n-1}{2}}\lambda^{n-2-\wt k}\le (1 + \lambda)^{\ell - \frac{n-1}{2}}(1 + r)^{\ell - \frac{n-1}{2}-\sigma}\lambda^{n-2-\wt k} \le C(1 + r)^{\wt k-\sigma-\frac{n-1}{2}}\lambda^{\frac{n-3}{2} -(\wt k-\ell)}.
\end{equation}
by Cauchy-Schwarz. Recalling our conditions on $\sigma$, we see that $\wt k - \sigma - \frac{n-1}{2} < 0$. Therefore, the corresponding term in \eqref{Lq_est} is bounded by a constant times
\[\lambda^{\frac{n-3}{2}-(\wt k-\ell)}\max\{1,\lambda^{(\wt k-\ell) - \frac{n-1}{2}}\} = \max\{\lambda^{\frac{n-3}{2} - (\wt k - \ell)},\lambda^{-1}\}\]
uniformly in $r$. Maximizing over the possible combinations of $\ell,m$ with $\ell + m = \wt k$, we have that 
\begin{equation}\label{right_var}
\lpn{\rho^{-\sigma}\partial_\lambda^{\wt k}R_{0,j}\rho^{-\sigma}(r,\cdot)}{q} \le C\max\{\lambda^{n-2-\wt k},\lambda^{\frac{n-3}{2}}\}
\end{equation}
for some $C > 0$, uniformly in $r$. A similar argument gives
\begin{equation}\label{left_var}
\lpn{\rho^{-\sigma}\partial_\lambda^{\wt k}R_{0,j}\rho^{-\sigma}(\cdot,s)}{q}\le C\max\{\lambda^{n-2-\wt k},\lambda^{\frac{n-3}{2}}\}
\end{equation}
uniformly in $s.$

For the final factor in the definition of $A$, which has asymmetric weights, we only need an estimate in the left variable in order to apply \lemref{young}. By \propref{Lq_nonimaginary} we have, for $1\le q\le \frac{n}{n-2},$
\begin{align*}
  &  \lpn{\rho^{-\sigma}g\partial_\lambda^{\wt k}R_{0,j}\rho^{-\alpha}(\cdot,s)}{q} \\
  & \hspace{1.5cm} \le C\rho^{-\alpha}(s) \lambda ^{n-2-\wt k}\sum\limits_{\ell+m = \wt k}\left[(1 + \lambda s)^{\ell - \frac{n-1}{2}}\lp C_1 + C_2\lambda^{m-\frac{n-1}{2}}\rp\right].
\end{align*}
We may repeat our previous argument almost exactly in order to bound this quantity. The only difference here is that the analogue of \eqref{nonuniform} has a factor of $\rho^{-\alpha}$ instead of $\rho^{-\sigma}.$ So in order to obtain an estimate which is uniform in $s$, we must enforce the condition that $\alpha \ge \max\{k - \frac{n-1}{2},0\}$ and recall that $\wt k \le k.$ Aside from this, the rest of the argument is identical, and so we have
\begin{equation}\label{R0_uniform}
\lpn{\rho^{-\sigma}f\partial_\lambda^{\wt k}R_{0,j}\rho^{-\alpha}(\cdot,s)}{q}\le C\max\{\lambda^{n-2-\wt k},\lambda^{\frac{n-3}{2}}\}
\end{equation}
uniformly in $s$, provided that $\alpha \ge \max\{k - \frac{n-1}{2},0\}.$

 We can now iteratively apply \lemref{young} to $\|\partial_\lambda^{k_1}A(\cdot,s)\|_{L^2}$. To do this, we must choose $q = \frac{2M}{2M-1}$ so that $\frac{M}{q} = \frac{1}{2}+(M-1)$. We also require $1\le q \le \frac{n}{n-2}$, which is equivalent to taking $M \ge \frac{n}{4}$. This then implies that we must take $\sigma >  \frac{n(2M-1)}{2M}+k_1$ in order for \propref{Lq_imaginary} and \propref{Lq_nonimaginary} to apply. In particular, we can take $M = \left\lceil\frac{n}{4}\right\rceil$, the smallest integer larger than $\frac{n}{4}$. Using \eqref{RV_sigmacondition}, we see that 
\[\sigma > 4\left\lceil\frac{n}{4}\right\rceil - 2 + k = n\lp\frac{4M-2}{n}\rp + k \ge n\lp\frac{4M-2}{4M}\rp+k \ge n\lp\frac{2M-1}{2M}\rp+k_1,\]
and so the following argument holds under this condition on $\sigma.$ Repeatedly applying \lemref{young} to $\ltwo{\partial_\lambda^{k_1} A(\cdot,s)}$ and using that $f$ is uniformly bounded, we obtain
\[\ltwo{\partial_\lambda^{k_1}A(\cdot,s)} \le C\lambda^{M(n-2)}. \]
The analogous estimate for $\|\partial_\lambda^{k_2}A(\cdot, r)\|$ combined with \eqref{pairing_est} gives
\begin{equation}\label{remainder_highfreq}
\left|\partial_\lambda^k\la(\rho^{-\sigma}R_{V,j}\rho^{-\sigma})A(\cdot,s), \overline{A(\cdot,r)}\ra_{L^2} \right| \le C\lambda^{2M(n-2)-1}
\end{equation}
for $\lambda \ge 1$, and this estimate holds uniformly in $r$ and $s.$

Next, we consider the remainder term in \eqref{Birman-Schwinger} when $0 < \lambda\le 1.$ In this case, taking the imaginary part in the left-hand side of \eqref{pairing_est} is essential, so we must estimate
\begin{equation}
\partial_\lambda^k\Imag\langle \rho^{-\sigma} R_{V,j}\rho^{-\sigma}A(\cdot,s),A(\cdot,r)\rangle_{L^2}.
\end{equation}
First, we note that the above can be written as a finite linear combination of terms where the imaginary part falls on either $R_{V,j}$ or at least one of the factors of $A$. Thus, by \eqref{eqn:RV_L2bd} and \eqref{RV_imaginarypart}, we can write
\begin{align}\label{pairing_imag}
\begin{split}
\left|\partial_\lambda^k\Imag\langle \rho^{-\sigma} R_{V,j}\rho^{-\sigma}A(\cdot,s),A(\cdot,r)\rangle_{L^2} \right|&\le C\max\limits_{k_1+k_2 + k_3\le k}\lambda^{n-2-k_3}\|\partial_\lambda^{k_1}A(\cdot,s)\|_{L^2}\|\partial_\lambda^{k_2}A(\cdot,r)\|_{L^2}\\
& + C\max\limits_{k_1+k_2\le k}\|\partial_\lambda^{k_1}\Imag A(\cdot,s)\|_{L^2}\|\partial_\lambda^{k_2}A(\cdot,r)\|_{L^2}
\end{split}
\end{align}
for $0 < \lambda \le 1.$ To estimate the first term on the right-hand side of \eqref{pairing_est}, we can argue analogously to the $\lambda \ge 1$ case, but now we use the low-frequency estimates from \propref{Lq_nonimaginary}, which give
\[\|\rho^{-\sigma}\partial_\lambda^{\wt k}R_{0,j}\rho^{-\sigma}\|_{L^q}\le C_1\lambda^{-\wt k}\rho^{-\sigma}(r) + C_2\lambda^{n-2-\wt k}(1 + \lambda r)^{\wt k - \frac{n-1}{2}}\rho^{-\sigma}(r) \le C\lambda^{-\wt k}\]
for any $\wt k \le k$ and $1\le q\le \frac{n}{n-2}$ as before. Similarly, we have
\begin{equation}\label{R0_nonim_term}
\|\rho^{-\sigma}\partial_\lambda^{\wt k}R_{0,j}\rho^{-\alpha}\|_{L^q} \le C\lambda^{-\wt k},
\end{equation}
for $\alpha \ge \max\{k-\frac{n-1}{2},0\}.$ Therefore, using \lemref{young}, we have that 
\[\|\partial_\lambda^{\wt k}A(\cdot,s)\|_{L^2} \le C\lambda^{-\wt k} \]
uniformly in $s$, for $0 < \lambda \le 1.$ Therefore, we have
\begin{equation}\label{pairing_1}
\max\limits_{k_1+k_2+k_3\le k}\lambda^{n-2-k_3}\|\partial_\lambda^{k_1}A(\cdot,s)\|_{L^2}\|\partial_\lambda^{k_2}A(\cdot,r)\|_{L^2} \le C\lambda^{n-2-k}.
\end{equation}
Now, to handle the second term on the right-hand side of \eqref{pairing_imag}, we note that one may expand $\Imag A(\cdot,s)$ into a linear combination of terms in which the imaginary part falls on at least one factor of $R_{0,j}.$ Therefore, we can use \propref{Lq_imaginary} to obtain that 
\begin{equation}\label{R0_im_term}
\|\rho^{-\sigma}\partial_\lambda^{\wt k}\Imag R_{0,j}\rho^{-\sigma}\|_{L^q} \le C\lambda^{n-2}(1 + \lambda r)^{k-\frac{n-1}{2}}\rho^{-\sigma} \le C\lambda^{n-2}
\end{equation}
for any $\wt k\le k.$ Thus, applying \lemref{young} in combination with \eqref{R0_nonim_term} and \eqref{R0_im_term} gives 
\[\|\partial_\lambda^{\wt k}\Imag A(\cdot,s)\|\le C\lambda^{n-2-\wt k}\]
for any $\wt k \le k.$ Hence, we have
\begin{equation}\label{pairing_2}
\max\limits_{k_1+k_2\le k}\|\partial_\lambda^{k_1}\Imag A(\cdot,s)\|_{L^2}\|\partial_\lambda^{k_2}A(\cdot,r)\|_{L^2} \le C\lambda^{n-2-k}
\end{equation}
uniformly in $r,s.$ Combining \eqref{pairing_1} and \eqref{pairing_2} with \eqref{pairing_imag} yields
\begin{equation}\label{remainder_lowfreq}
\partial_\lambda^k\Imag\langle \rho^{-\sigma} R_{V,j}\rho^{-\sigma}A(\cdot,s),A(\cdot,r)\rangle_{L^2} \le C\lambda^{n-2-k}, 
\end{equation}
and so the remainder in the Birman-Schwinger expansion of $R_{V,j}$ satisfies the claimed estimate for $0<\lambda\le 1.$

%

Now we consider a generic term in the sum in \eqref{Birman-Schwinger} for $1\le \ell \le 2M-1$. As before, we use the fact that $V = \rho^{-2\sigma}(r)f(r)$ for some $f\in L^\infty(\R^+)$ to write
\begin{align}
\begin{split}
\rho^{-\alpha}\partial_\lambda^k R_{0,j}(VR_{0,j})^\ell\rho^{-\alpha}(r,s)  = \partial_\lambda^k (\rho^{-\alpha}R_{0,j}\rho^{-\sigma})(\rho^{-\sigma}fR_{0,j}\rho^{-\sigma})^{\ell - 1}(\rho^{-\sigma}fR_{0,j}\rho^{-\alpha})(r,s).
\end{split}\end{align}
Writing the above as an $L^2$-pairing, we have
\begin{equation}\label{L2_pairing}\rho^{-\alpha} \partial_\lambda^kR_{0,j}(VR_{0,j})^\ell\rho^{-\alpha}(r,s)  =\partial_\lambda^k \la\rho^{-\alpha}R_{0,j}\rho^{-\sigma}(r,\cdot),(\rho^{-\sigma}fR_{0,j}\rho^{-\sigma})^{\ell-1}(\rho^{-\sigma}fR_{0,j}\rho^{-\alpha})(\cdot,s)\ra_{L^2}.
\end{equation}
Upon taking the imaginary part, we obtain a finite linear combination of terms of the form \eqref{L2_pairing} where at least one factor of $R_{0,j}$ has the imaginary part acting on it. We assume without loss of generality that the leftmost factor on the right-hand side of \eqref{L2_pairing} has the imaginary part, and thus we can apply H\"older's inequality to obtain
\begin{align}\label{holder}
\begin{split}
\left|\partial_\lambda^k \la\rho^{-\alpha}\Imag R_{0,j}\rho^{-\sigma}(r,\cdot),(\rho^{-\sigma}fR_{0,j}\rho^{-\sigma})^{\ell-1}(\rho^{-\sigma}fR_{0,j}\rho^{-\alpha})(\cdot,s)\ra_{L^2}\right|\\ 
& \hspace{-3.7in}\le C\|\rho^{-\alpha}\partial_\lambda^{k_1}\Imag R_{0,j}\rho^{-\sigma}(r,\cdot)\|_{L^{q'}}\\
& \hspace{-3.5in}\times\|(\rho^{-\sigma}f\partial_\lambda^{k_2}R_{0,j}\rho^{-\sigma})\dotsm(\rho^{-\sigma}f\partial_\lambda^{k_{\ell-1}}R_{0,j}\rho^{-\sigma})(\rho^{-\sigma}f\partial_\lambda^{k_\ell}R_{0,j}\rho^{-\alpha})(\cdot,s)\|_{L^p}
\end{split}
\end{align}
for some $1\le q' < \infty$ to be determined, and $p$ given by $\frac{1}{q'} + \frac{1}{p} = 1,$ where $k_1 + k_2 +\dotsm + k_\ell = k.$ 

If $0 < \lambda \le 1$, we recall that by \propref{Lq_imaginary},
\begin{equation}\label{imag_Lq}
\|\rho^{-\alpha}\partial_\lambda^{k_1}\Imag R_{0,j}\rho^{-\sigma}(r,\cdot)\|_{L^{q'}} \le C\lambda^{n-2}
\end{equation}
provided that $\sigma > \frac{n}{q'} + k$ and $\alpha \ge \max\{ k_1 - \frac{n-1}{2},0\}$. Similarly, for any $\wt k\le k$, we have by \propref{Lq_nonimaginary} that for any $1\le q\le \frac{n}{n-2}$,
\begin{equation}\label{nonimag_Lq}
\|\rho^{-\sigma}\partial_\lambda^{\wt k}R_{0,j}\rho^{-\sigma}(r,\cdot)\|_{L^q} \le C\lambda^{-\wt k},
\end{equation}
if $\sigma > \frac{n}{q}+k$, along with the analogous estimate when the norm is taken with respect to the other variable.
Using \eqref{imag_Lq}, \eqref{nonimag_Lq}, and repeated applications of \lemref{young} to the right-hand side of \eqref{holder}, we obtain
\begin{equation}\label{nonremainder_lowfreq}
 \left|\partial_\lambda^k \la\rho^{-\alpha}\Imag R_{0,j}\rho^{-\sigma}(r,\cdot),(\rho^{-\sigma}fR_{0,j}\rho^{-\sigma})^{\ell-1}(\rho^{-\sigma}fR_{0,j}\rho^{-\alpha})(\cdot,s)\ra_{L^2}\right|\le C\lambda^{n-2 - k}
\end{equation}
when $0 < \lambda \le 1,$ as long as we choose $q,q'$ such that $\frac{1}{q'} + \frac{\ell}{q} = \ell$ and provided that $\sigma >\max\{\frac{n}{q'},\frac{n}{q}\}+k$. Since $\frac{n}{q} \ge n-2-k$, we have by \eqref{RV_sigmacondition} that
\[\sigma > 4\left\lceil\frac{n}{4}\right\rceil - 2 + k \ge n-2+k,\]
and so we can ensure that $\sigma > \frac{n}{q} + k$ if $q$ is chosen sufficiently close to, but just below $\frac{n}{n-2}.$ Given this choice of $q,$ we also have that $q'$ lies just above $\frac{n}{2\ell}$, and so 
\[\sigma > 4\left\lceil\frac{n}{4}\right\rceil - 2+k = 2(2M-1) + k,\]
and since $1\le \ell \le 2M-1$, we can ensure that $\sigma > \frac{n}{q'} + k$, since $\frac{n}{q'}+k$ can be made arbitrarily close to $2\ell + k.$ Therefore, under the claimed conditions on $\sigma$ and $\alpha$, we have that 
\[\left| \partial_\lambda^kR_{0,j}(VR_{0,j})^\ell\right| \le C\lambda^{n-2-k}\]
for $0 < \lambda\le 1$ and any $1\le \ell \le 2M-1.$

In the case where $\lambda \ge 1$, we have that
\[\|\rho^{-\alpha}\partial_\lambda^{k_1}\Imag R_{0,j}\rho^{-\sigma}(r,\cdot)\|_{L^{q'}} \le C\max\{ \lambda ^{n - 2- k_1-\frac{n}{q'}},\lambda^{\frac{n-3}{2}}\}\]
uniformly in $r$ as before, provided that $\alpha \ge \max\{k-\frac{n-1}{2},0\}$ and $\sigma > \frac{n}{q'} + k$. Next, choose some $q$ which lies just below $\frac{n}{n-2}$ as above. Then, for any $\wt k \le k,$
\[\|\rho^{-\sigma}\partial_\lambda^{\wt k}R_{0,j}\rho^{-\sigma}(r,\cdot)\|_{L^q}\le C\max\{\lambda^{n-2 - \wt k},\lambda^{\frac{n-3}{2}}\}, \]
uniformly in $r$, provided that $\sigma > \frac{n}{q} + k$, along with the analogous estimate when the $L^q$ norm is taken over the second variable. We also note that for the rightmost factor in \eqref{holder}, we have
\[\|\rho^{-\sigma}\partial_\lambda^{\wt k}R_{0,j}\rho^{-\alpha}(\cdot,s)\|_{L^q}\le C\max\{\lambda^{n-2-\wt k},\lambda^{\frac{n-3}{2}}\},\]
uniformly in $s,$ provided $\alpha \ge \max\{k - \frac{n-1}{2},0\}.$ Given these estimates and \lemref{young}, we can maximize over the possible combinations of $k_1,\dotsc,k_\ell$ to see from \eqref{holder} that
\begin{equation}\label{nonremainder_highfreq}
\begin{split}
\left|\partial_\lambda^k \la\rho^{-\alpha}\Imag R_{0,j}\rho^{-\sigma}(r,\cdot),(\rho^{-\sigma}fR_{0,j}\rho^{-\sigma})^{\ell-1}(\rho^{-\sigma}fR_{0,j}\rho^{-\alpha})(\cdot,s)\ra_{L^2} \right|\\
& \hspace{-2in}\le \lambda^{\ell( n-2)-k}\max\{\lambda^{n-2-\frac{n}{q'}},\lambda^{\frac{n-3}{2}}\}
\end{split}
\end{equation}
provided that $\frac{1}{q'} + \frac{\ell}{q} = \ell$ and $\sigma >\max\{\frac{n}{q'},\frac{n}{q}\}+k$. As shown previously, this condition on $\sigma$ is satisfied under the hypothesis \eqref{RV_sigmacondition} if $q$ is chosen close enough to $\frac{n}{n-2}.$ Furthermore, for this choice of $q$, we have that $q'$ lies just above $\frac{n}{2\ell}$. We claim that this implies that the the bound \eqref{nonremainder_highfreq} is smaller than the estimate \eqref{remainder_highfreq}. To see this, note that if $q'$ is chosen sufficiently close to $\frac{n}{2\ell}$, then $\frac{n}{q'} = 2\ell + \ve$ for some $\ve > 0$. Then, we have
\[\ell(n-2) + n-2 - \frac{n}{q'} = \ell(n-2) + (n-2) - 2\ell -\ve \le \ell(n-4) + (n-2).\]
If $n\le 4$, then the above is smaller than $n-2$ for all $\ell = 1,\dotsc,2M-1$. If $n\ge 4,$ then we have
\[\ell(n-4)+(n-2) \le (2M-1)(n-4) +(n-2) = 2M(n-2) -2(2M-1) \le 2M(n-2) - 1.\]
Furthermore, we note that 
\[\ell(n-2) + \frac{n-3}{2} \le (2M-1)(n-2) - \frac{n-3}{2} \le 2M(n-2) - \frac{n-1}{2} \le 2M(n-2) - 1.\]
Therefore, the exponent on $\lambda$ in \eqref{nonremainder_highfreq} is smaller than that of \eqref{remainder_highfreq} for any $\ell,$ and hence we have
\begin{equation}\label{nonremainder_highfreq2}
\left|\partial_\lambda^k \la\rho^{-\alpha}\Imag R_{0,j}\rho^{-\sigma}(r,\cdot),(\rho^{-\sigma}fR_{0,j}\rho^{-\sigma})^{\ell-1}(\rho^{-\sigma}fR_{0,j}\rho^{-\alpha})(\cdot,s)\ra_{L^2} \right|\le \lambda^{2M(n-2)-1}
\end{equation}
when $\lambda \ge 1$.

Now, if the imaginary part falls on any factor other than the first on the right-hand side of \eqref{L2_pairing}, we simply repeat the preceding argument, but with the $L^{q'}$ norm on that factor.

Finally, we consider the case where $\ell = 0$ in \eqref{Birman-Schwinger}. For this term, we must simply obtain pointwise bounds on $\rho^{-\alpha}(r)\partial_\lambda^k\Imag R_{0,j}(r,s)\rho^{-\alpha}(s)$. Recall that by \lemref{stone}, we have
\[\Imag R_{0,j}(r,s) = \frac{\pi }{2} \lambda^{n-2}(\lambda r\lambda s)^{-\frac{n-2}{2}} J_{\nu_j}(\lambda r)J_{\nu_j}(\lambda s).\]
Therefore, $\partial_\lambda^k R_{0,j}(r,s)$ can be written as a finite linear combination of terms of the form 
\begin{equation}\label{R0_linearcomb}
\lambda^{n-2-k}(\lambda r)^{\ell-\frac{n-2}{2}}(\lambda s)^{m - \frac{n-2}{2}}J_{\nu_j+\alpha}(\lambda r)J_{\nu_j+\beta}(\lambda s)
\end{equation}
for $\ell + m = k$, $|\alpha|\le \ell$, and $|\beta| \le m$. Using the standard asymptotics of the Bessel functions, we have that the above is bounded in absolute value by a constant times 
\begin{equation}\label{R0_unweighted}
\lambda^{n-2-k}(1 + \lambda r)^{\ell - \frac{n-1}{2}}(1 + \lambda s)^{m - \frac{n-1}{2}}.
\end{equation}
Next, we note that 
\[\rho^{-\alpha}(r)(1 + \lambda r)^{\ell - \frac{n-1}{2}} \le C(1 + \lambda)^\ell,\]
for all $\lambda,$ uniformly in $r$, under the assumption that $\alpha \ge \max\{k - \frac{n-1}{2},0\}$. The analogous estimate holds for $\rho^{-\alpha}(s)(1 + \lambda s)^{m-\frac{n-1}{2}}$, and therefore, we have that 
\begin{equation}\label{R0_pointwisebound}
\left|\rho^{-\alpha}(r)\partial_\lambda^k R_{0,j}(r,s)\rho^{-\alpha}(s)\right| \le C\lambda^{n-2-k}(1 + \lambda)^k.
\end{equation}
Combining \eqref{R0_pointwisebound} with \eqref{remainder_lowfreq}, \eqref{remainder_highfreq}, \eqref{nonremainder_lowfreq}, and \eqref{nonremainder_highfreq}, the proof of \propref{RV_pointwise} is complete.
\end{proof}

\section{Dispersive estimates}
\label{dispersive}

In this section, we prove the main estimate in \thmref{disp_est}. To accomplish this, we write the spectral measure for $-\Delta_{C(X)} + V$  as 
\[d\Pi_V(\lambda;x,y) = \frac{1}{\pi i}[R_V(\lambda^2+i0;x,y) - R_V(\lambda^2 - i0;x,y)]\lambda\,d\lambda = \frac{1}{\pi}\Im R_V(\lambda^2+i0;x,y)\lambda\,d\lambda.\]
 Then, we can write
\[\left[e^{it(-\Delta_{C(X)} + V)}P_c\right](x,y) = \int\limits_0^\infty e^{it\lambda^2}d\Pi_V(\lambda;x,y)= \frac{1}{\pi }\int\limits_0^\infty e^{it\lambda^2}\Imag R_{V}(\lambda^2+i0;x,y) \lambda\,d\lambda,\]
where we recall that $P_c$ denotes projection onto the continuous spectrum of $-\Delta_{C(X)}+V.$ Projecting further onto the span of $\varphi_j$, we obtain
\begin{equation}
\left[e^{it(-\Delta_{C(X)} + V)}P_cE_j\right](r,s) = \frac{1}{\pi}\int\limits_0^\infty e^{it\lambda^2}\Im R_{V,j}(\lambda^2+i0;r,s)\lambda\,d\lambda
\end{equation}
since $V$ is radial. Therefore, the estimate in \thmref{disp_est} is equivalent to
\begin{equation}\label{spectral_measure_est}
\left|\frac{1}{\pi }\int\limits_0^\infty e^{it\lambda^2} \rho^{-\alpha}(r)\Im R_{V,j}(\lambda^2+i0;r,s)\rho^{-\alpha}(s)\lambda\,d\lambda\right| \le Ct^{-\frac{n}{2}}
\end{equation}
for $\alpha > 2\left\lceil\frac{n}{4}\right\rceil(n-2) - \frac{n-1}{2} + 2$.

\begin{proof}[Proof of \thmref{disp_est}]
Assume that $n$ is odd, and let $\chi\in C_0^\infty(\R)$ be a cutoff function which is identically one on $[-1/2,1/2]$ and zero outside $[-1,1]$. We then consider the low-frequency component of the left-hand side of \eqref{spectral_measure_est}, given by 
\begin{equation}\label{low_frequency}
 \frac{1}{\pi }\int\limits_0^\infty e^{it\lambda^2} \,\chi(\lambda)\rho^{-\alpha}(r)\Im R_{V,j}(\lambda^2+i0;r,s)\rho^{-\alpha}(s)\lambda\,d\lambda.
\end{equation}
 Noting that the operator $\frac{1}{2it\lambda}\partial_\lambda$ preserves $e^{it\lambda^2}$, we may integrate by parts $N = \frac{n-1}{2}$ times in $\lambda$ to obtain
\begin{equation}\label{IBP}
\frac{C_N}{t^N}\int\limits_0^\infty e^{it\lambda^2}\partial_\lambda\lp\frac{1}{\lambda}\partial_\lambda\rp^{N-1}\left[\chi(\lambda)\rho^{-\alpha}(r)\Im R_{V,j}(\lambda^2+i0;r,s)\rho^{-\alpha}(s)\right]\,d\lambda
\end{equation}
for some $C_N\in\C\setminus 0.$ That no boundary terms appear at $\lambda = 0$ follows from the fact that all derivatives of $\chi(\lambda)$ vanish identically near the origin and that $\rho^{-\alpha}(r)\partial_\lambda^k\Im R_{V,j}(\lambda^2+i0)\rho^{-\alpha}(s)$ vanishes to order $n-2-k$. To be more precise, all boundary terms at $\lambda = 0$ must involve factors of the form
\[\lp\frac{1}{\lambda}\partial_\lambda\rp^k\left[\chi(\lambda)\rho^{-\alpha}(r)\Imag R_{V,j}(\lambda^2+i0)\rho^{-\alpha}(s)\right]\] 
for some $0 \le k \le N-1 = \frac{n-3}{2}$. If any derivatives fall on $\chi$, then the corresponding term obviously vanishes at $\lambda = 0$. If instead, all derivatives fall on $\Im R_{V,j}$, then by \propref{RV_pointwise}, the corresponding term is bounded by a constant times $\lambda^{n-2-2k}\le \lambda^{n-2-(n-3)} = \lambda$, which vanishes at $\lambda = 0.$ Therefore, all boundary terms are necessarily zero.

 Now, observe that when expanding the integrand in \eqref{IBP} via the product rule, any terms in which a derivative falls on the factor of $\chi(\lambda)$ can be written as 
\[C_Nt^{-N}\int\limits_{0}^\infty e^{it\lambda^2}G(\lambda;r,s)\,d\lambda\]
for some $G(\lambda;r;s)$ which is smooth and compactly supported away from $0$ in $\lambda,$ and bounded uniformly in $r,s$ by \propref{RV_pointwise}. Applying the standard dispersive estimate for the Schr\"odinger equation on $\R,$ we have
\begin{equation}\label{G_terms}
t^{-N}\left|\int\limits_{0}^\infty e^{it\lambda^2}G(\lambda;r,s)\,d\lambda\right|\le Ct^{-N-\tfrac{1}{2}}\|\widehat G(\cdot\,;,r,s)\|_{L^1},
\end{equation}
where $\widehat G$ denotes the Fourier transform in $\lambda$ (extend $G$ by zero to a function on $\R$ to compute this Fourier transform). Since $n$ is odd, we may choose $N = \frac{n-1}{2}$, so the right-hand side of \eqref{G_terms} is bounded by $C t^{-\frac{n}{2}}$ as claimed, after possibly increasing $C$.

Now, any terms obtained from expanding \eqref{IBP} where no derivatives fall on the factor of $\chi$ must be of the form 
\begin{equation}\label{lowfreq_linearcomb}
\lambda^{1-2N + k}\chi(\lambda)\rho^{-\alpha}(r)\partial_\lambda^k\Im R_{V,j}(\lambda^2+i0;r,s)\rho^{-\alpha}(s)
\end{equation}
for some $k = 1,2,\dotsc,N$ since at least one derivative always falls on the factor of $\Im R_{V,j}.$ By \propref{RV_pointwise}, we have that each of the above terms is bounded in absolute value by a constant times 
\[\lambda^{1-2N + k}\chi(\lambda)\lambda^{n-2-k} = \lambda^{n-1 - 2N}\chi(\lambda)\]
uniformly for $r,s >0.$ For our choice of $N = \frac{n-1}{2}$, we have that $\lambda^{n-1 - 2N} = 1$, and hence \eqref{lowfreq_linearcomb} is a smooth function of $\lambda$, and so its Fourier transform is bounded in $L^1$. Once again, using the standard $L^1\to L^\infty$ dispersive estimate for the free one-dimensional Schr\"odinger equation, we have that 
\begin{align}\label{low_frequency_bound}
\begin{split}
t^{-N}\left|\int\limits_{0}^\infty e^{it\lambda^2}\lambda^{1-2N + k}\chi(\lambda)\rho^{-\alpha}(r)\partial_\lambda^k\Im R_{V,j}(\lambda^2+i0)(r,s)\rho^{-\alpha}(s)\,d\lambda\right| \le C t^{-N-\frac{1}{2}} = Ct^{-\frac{n}{2}},
\end{split}
\end{align}
uniformly in $r,s.$ We remark that it is in this calculation that the choice of $N = \frac{n-1}{2}$, and hence the power of $t^{-\frac{n}{2}}$, cannot be improved, since any additional derivatives which fall on $\Im R_{V,j}(\lambda^2+i0)$ would yield an integrand which is not bounded smooth near $\lambda = 0$. We also note that for this portion of the argument, we only require that $\alpha \ge 0$, since we did not differentiate $\Im R_{V,j}$ more than $\frac{n-1}{2}$ times.

Next, we consider the ``high-frequency" component of \eqref{spectral_measure_est}, which we define by 
\begin{equation}\label{high_frequency}
\frac{1}{\pi}\int\limits_{0}^\infty e^{it\lambda^2}\chi(\lambda/R)(1 - \chi(\lambda)) \rho^{-\alpha}(r)\Im R_{V,j}(\lambda^2+i0;r,s)\rho^{-\alpha}(s)\,\lambda\,d\lambda
\end{equation}
for any $R\in[1,\infty)$. To control this term, we integrate by parts as before to obtain
\begin{align*}
&\frac{1}{\pi}\int\limits_{0}^\infty e^{it\lambda^2}\chi(\lambda/R)(1 - \chi(\lambda)) \rho^{-\alpha}(r)\Im R_{V,j}(\lambda^2+i0;r,s)\rho^{-\alpha}(s)\,\lambda\,d\lambda\\
& \hskip .05in = C_N t^{-N}\int\limits_{0}^\infty e^{it\lambda^2}\partial_\lambda\lp\frac{1}{\lambda}\partial_\lambda\rp^{N-1}\left[\chi(\lambda/R)(1 - \chi(\lambda))\rho^{-\alpha}(r)\Im R_{V,j}(\lambda^2+i0;r,s)\rho^{-\alpha}(s)\right]\,d\lambda
\end{align*}
for any $N > 0$ and some corresponding constant $C_N$. We aim to show that the integrand can be bounded uniformly in $L^1(\R,d\lambda)$ as $R\to\infty.$ We also claim that it is sufficient to consider the case where all the derivatives in $\lambda$ fall on the factor of $\Imag R_{V,j}$. To see this, note that $\partial_\lambda(1-\chi(\lambda))$ is supported in a fixed compact set which is bounded away from $\lambda = 0,$ and that $\partial_\lambda\chi(\lambda/R) = \frac{1}{R}\chi'(\lambda/R)$ is supported away from $\lambda = 0$ in a set of size $\mathcal O(R).$ Therefore, we need only show that 
\begin{equation}\label{high_freq_integrand}\frac{1}{\lambda^{N-1}}\chi(\lambda/R)(1-\chi(\lambda))\rho^{-\alpha}(r)\partial_\lambda^N\Imag R_{V,j}(\lambda^2+i0;r,s)\rho^{-\alpha}(s)
\end{equation}
has bounded $L^1$ norm, and that the estimate is uniform with respect to $r,\,s,$ and $\,R$. For this, we utilize \propref{RV_pointwise}, which implies that if $\alpha \ge \max\{N -\frac{n-1}{2},0\}$ and $\sigma > 4\left\lceil\frac{n}{4}\right\rceil -2 + N$, then \eqref{high_freq_integrand} is bounded by a constant times $\langle\lambda\rangle^{1-N+L},$ uniformly in $r,\,s,$ and $R$, where $L = 2\left\lceil\frac{n}{4}\right\rceil(n-2)-1.$ Thus, choosing
\[N = 2\left\lceil\frac{n}{4}\right\rceil(n-2)+2,\]
guarantees that \eqref{high_freq_integrand} is uniformly bounded in $L^1$. Noting that $N\ge \frac{n}{2}$ if $N$ is chosen as above, we obtain
\begin{equation}\label{high_frequency_bound}
\lim\limits_{R\to\infty}\left|\frac{1}{\pi}\int\limits_{0}^\infty e^{it\lambda^2}\chi(\lambda/R)(1 - \chi(\lambda)) \Im R_{V,j}(\lambda^2+i0;r,s)\,\lambda\,d\lambda\right|\le Ct^{-\frac{n}{2}},
\end{equation} 
where $C > 0$ is independent of $r,\,s$ and $R.$ Our choice of $N$ also determines the maximum number of derivatives of $\Im R_{V,j}(\lambda^2+i0;r,s)$ that must be taken, which yields 
\[\sigma > 4\left\lceil\frac{n}{4}\right\rceil -2 + 2\left\lceil\frac{n}{4}\right\rceil(n-2)+2 = 2n\left\lceil\frac{n}{4}\right\rceil \]
as the sufficient condition on the decay rate of $V$. Also, the condition on $\alpha$ becomes 
\[\alpha > N - \frac{n-1}{2} = 2\left\lceil\frac{n}{4}\right\rceil(n-2) +2 - \frac{n-1}{2},\]
as stated in \thmref{disp_est}. Under these conditions on the weights, we can combine \eqref{low_frequency_bound} and \eqref{high_frequency_bound} to obtain \eqref{spectral_measure_est}, which completes the proof of \thmref{disp_est}.
\end{proof}

\begin{rmk}\textnormal{
In the case where $n$ is even, we find that in repeating the argument just prior to \eqref{low_frequency_bound}, the largest $N$ we can choose is $\frac{n-2}{2}$, which leads to a decay rate of $t^{-\frac{n-1}{2}}$ in the $L^1\to L^\infty$ estimate. The remainder of the argument goes through without modification, yielding \eqref{disp_est_even}.
}
\end{rmk}

\appendix


\section{Absence of embedded resonances on product cones}
\label{app:embres}

\subsection{Radial potentials and ODE methods}

In this appendix, we establish the absence of embedded resonances and eigenvalues result claimed in \thmref{absence.thm} using the results from Theorem XIII.56 of \cite{RSv4}.  To prove this, we need to use the fact that $V$ is decaying to treat $Vu$ as a perturbative term in the limit and prove that if $u \in L^2$, then hence $Vu$ is perturbative and we can write
\[
u_\pm (r)  = \frac{e^{ \pm i \lambda r}}{r^{\frac{n-1}{2}}} \mp \int_r^\infty   \frac{\sin (\lambda (r-s)) }{ \lambda (rs)^{\frac{n-1}{2}}} V (s) u_\pm (s) ds,
\]
where by $u_\pm$ we mean the outgoing/incoming functions converging to the Jost solution asymptotic of the form
\[
 \frac{e^{ \pm i \lambda r}}{r^{\frac{n-1}{2}}}.
 \]
This gives the exact integrability condition in Theorem XIII.56 of \cite{RSv4}.  It is also an integral equation that can also be solved using Picard iteration.  Using the variation of parameters formula to solve
\[
(-\Delta_{C(X)} -\lambda^2) u = V(r) u 
\]
for any $V$ with $\int_a^\infty V(r) dr < \infty$ for some $a > 0$, we can write
\begin{align*}
 & u(r) = c_+ u_+ (r) + c_- u_- (r)  \\
 & \hspace{1cm} + \int_r^\infty \frac{u_+ (r) u_- (s) }{W(s)}  V(s) u(s) s^{n-1} ds + \int_r^\infty \frac{u_- (r) u_+ (s) }{W(s)}  V(s) u(s) s^{n-1} ds,   
\end{align*}
where $W(s) = 2 s^{n-1}$.  Hence, if we have a resonance $u \in L^{2,\sigma}$ for $\sigma > \frac12$, we have that indeed we can see the integral terms on the right converge and hence derive a contradiction to the existence of resonances that are not eigenvalues.  

For radial potentials, we have observed that there are no embedded eigenvalues using the ODE based tools of Theorem XIII.56 of \cite{RSv4}, but the issue of absence of embedded resonances down to $\lambda = 0$ still must be established.  We state the result here.   

\begin{proposition}\label{res.crit.lemma}
For $V \in \rho^{-2 \sigma}L^\infty(\R^+)$ with $ \sigma > \frac12$, if $R_{V,j} (z^2)$ has a pole at 
$\tau$ for $\tau \in \RR\backslash\{0\}$, 
then $ \tau^2$ is an embedded eigenvalue for $-\Delta_{C(X)} + V$.
\end{proposition}

\begin{proof}[Proof of Theorem \ref{absence.thm}]

By \propref{res.crit.lemma}, it suffices to prove that $R_V(z^2)$ has no real poles corresponding to eigenvalues of $-\Delta_{C(X)}+V$ embedded in the continuous spectrum. It also suffices to prove this fact for $R_{V,j}(z^2)$, since the fact that $V$ is radial means that $R_V$ respects the decomposition into harmonics on the link. The method of proof from Theorem XIII.56 of \cite{RSv4} proves that for $V$ sufficiently decaying, a radial operator has no positive eigenvalues.  This relies on a formal analysis of Jost solutions, proving that no linear combination of them can be an $L^2$ function.   In particular, if we assume we have an embedded eigenvalue at energy $\lambda^2$ with corresponding radial eigenfunction $\phi$ on the $j$th harmonic, it satisfies
\[
\left(  \partial_r^2 + \frac{n-1}{r} \partial_r + \lambda^2 - \frac{\mu_j^2}{r^2} \right) \phi = 0 
\]
meaning that $\infty$ is a regular singular point.  As a result, as $r \to \infty$ the Jost solutions are of the form
\[
\phi_\pm (r) = e^{\pm i \lambda  r} r^{-\frac{n-1}{2}} + O(r^{\frac{n-3}{2}} ).
\]
Hence, linear combinations generically take the form
\[
A \sin (\lambda r + \theta )  r^{-\frac{n-1}{2}} ,
\]
which is easily seen to not be in $L^2 (C(X))$, which yields a contradiction. Further discussion of the Jost solutions can be computed as in \cite{regge1965potential}, Ch. 5.

\end{proof}

\subsection{Absence of Embedded Resonances for Non-Radial Potentials}

For the purposes of ruling out embedded resonances, a geometrically robust approach is to use a boundary pairing formula on radially compactified space, as in \cite[\S 2.3]{melrose1995geometric}, to prove that an embedded resonance is an embedded eigenvalue and hence has at least $L^2$ decay. The boundary pairing formula for Schr\"odinger operators on Euclidean space is derived from the observation that a pole of the resolvent corresponds to a solution to 
\[(-\Delta_{\R^n} +V - \lambda^2)u = f,\]
say for $f \in \mathcal{S}$ a Schwartz class function. The solution to this equation takes the form 
\[
u = e^{i \frac{\lambda}{x}} x^{\frac12 (n-1)} w_+ + e^{-i \frac{\lambda}{x}} x^{\frac12 (n-1)} w_- , \ \ w_\pm \in C^\infty (\mathbb{S}^n_+)
\]
where $x = \frac{1}{r}$ and $\mathbb{S}^n_+$ is the upper hemisphere of the sphere $\mathbb{S}^n$. The boundary pairing formula states that for solutions
 \[
(-\Delta + V - \lambda^2 ) u^{(\ell)} = f^{(\ell)}
\] 
with $\ell = 1,2$, we have
\[
2 i \lambda \int_{\mathbb{S}^{n-1}} (v_+^{(1)} \overline{v_+^{(2)}} - v_-^{(1)} \overline{v_-^{(2)}}) dz = \int_{\mathbb{R}^n} ( f^{(1)} \overline{u^{(2)}} + u^{(1)} \overline{f^{(2)}} ) dz
\]
with $v_{\pm}^{(\ell)} = w_{\pm}^{(\ell)} |_{\partial \mathbb{S}_+^n}$.  For a pole of the outgoing resolvent, we observe that $v_+$ vanishes identically.

Hence, any embedded resonance can be seen to be an embedded eigenvalue.  To eliminate embedded eigenvalues, we follow the work of Froese et al \cite{froese1982absence} to prove that $L^2$ eigenfunctions must in fact exhibit super-polynomial decay.  The arguments there involve constructing a series of positive commutator arguments to obtain this rapid decay.  To begin, take $\varepsilon, \gamma  > 0$ and define $\rho(|x|) = \langle x \rangle$.  Then, the function 
\[
F (x) = \gamma \ln ( \rho (1 + \varepsilon \rho)^{-1}).
\]
Then, $\nabla F = x g$ for $g = \gamma \rho^{-2} (1 + \varepsilon \rho)^{-1}$ and 
\[
(x \cdot \nabla)^2 g - (x \cdot \nabla )(\nabla F)^2 \leq 4 \gamma (\gamma +2) \rho^{-2}.
\]

Defining $\psi_R = e^F \psi$, a conjugated form of the equation can be written as a modified quadratic form.  Coupling this form with a Mourre estimate (positive commutator using $r \partial r$ (the radial version of $x \cdot \nabla$) we can prove that the set of polynomial weights for which $\psi \in L^2$ is open and can be extended to $\infty$.  The Mourre estimate serves as a means to construct the weak limit of the resolvent at the real axis.

Once super-polynomial behavior is established, a similar open set for exponential decay can be established using the assumption that $e^{\alpha_0 \rho} \rho^\lambda \psi \in L^2$ for all $\lambda$ and showing that this implies then that $e^{(\alpha_0 + \gamma) \rho} \psi \in L^2$.  To do this, build the function
\[
F(x) = \alpha_0 \rho + \lambda \ln (1 + \gamma \rho/\lambda)
\]
and derive a similar contradiction.  

Once super-exponential decay is established, the strategy of Vasy-Wunsch \cite{VWsurexp} can be applied to prove a unique continuation argument by conjugating the operator to
\[
P_\alpha = e^{\alpha r'} (-\Delta + V - \lambda) e^{-\alpha r'} 
\]
for $r'$ some smoothed version of $r$ to be determined.  Then, 
\[
0 = \| P_\alpha \phi_\alpha \|^2 = \| \Re P_\alpha \psi_\alpha \|^2 + \| \Im P_\alpha \psi_\alpha \|^2 + \langle i [ \Re P_\alpha, \Im P_\alpha] \psi_\alpha, \psi_\alpha \rangle.
\]
Hence, one uses that $ i [ \Re P_\alpha, \Im P_\alpha] $ is a positive commutator term.  
We require that
\[
[ \Delta, 2 \partial_r + (\partial_r  \log A)] \geq c \frac{1}{r^2} \Delta_{\theta} + R
\]
for $R$ in the calculus of first order conic vector fields.

\subsection{The Boundary Pairing Formula}

Following a suggestion of Dean Baskin, we can interpret embedded resonances for more general conic Schr\"odinger operators through the boundary pairing formula of Melrose.  

\subsubsection{Existence of the Boundary Pairing Formula on Cones}

We outline the necessary generalizations to the presentation of the Boundary Pairing formula from the book of Melrose \cite{melrose1995geometric}.  First, we must consider the radial compactification of a cone to a compactified manifold $\hat C = X \times [0,1]$ for $X$ the link of the cone, which is similar to that in \cite{baskin2019radiation}.   We again consider solutions of the equation
\[(-\Delta_{\R^n} +V - \lambda^2)u =f\]
for $f \in \mathcal{S}$ for instance say a Schwartz class function.

Let 
\[
u = u_+  + u_- ,
\]
namely a sum of the outgoing and incoming solutions,
with 
\begin{equation}
    \label{iosolns}
u_{\pm} = e^{\pm i \lambda/x} x^{\frac{n-1}{2}}  w_\pm
\end{equation}
for $w_{\pm} \in C^\infty (\hat C)$.  This formula appears in a variety of settings in the literature, starting with the foundational work of Melrose \cite{melrose1994spectral} on asymptotically Euclidean manifolds, then Melrose-Zworski \cite{MelroseZworski1996} on general scattering manifolds with smooth boundary, Hassell-Vasy \cite{hassellvasy1999} on scattering manifolds with conic points, and the corresponding discussion of Guillarmou-Hassell-Sikora  in \cite{guillarmou2013resolvent}, Section $5$.  See also \cite{baskin2019scattering,yang2020diffraction} for a recent discussion on product cones that contains formulae from which such a decomposition can be obtained.

Then, we claim that
\begin{equation}
    \label{bpformula}
2 i \lambda \int_X [  v^{(1)}_+ \overline{v^{(2)}_+} - v^{(1)}_- \overline{v^{(2)}_-} ] dv_h = \int_{C(X)}  [  f^{(1)} \overline{u^{(2)}} - u^{(1)} \overline{f^{(2)}} ] dr dv_h
\end{equation}
with $v^{(i)}_{\pm} = w^{(i)}_{\pm} |_X$.  We will need to consider behaviors both at $x=0$ and $1$ whereas on Euclidean space the compactification really only sees $\infty$.  See Ch. $2.3$ and Ch. $6$ of Melrose \cite{melrose1995geometric}.  Compactify via the stereographic projection to the quarter circle 
$$S^1_{+,+} = \{ (z_1, z_2 ) \in S^1 \subset \RR^2 | z_1 \geq 0, z_2 \geq 0 \}$$ 
with $r \to (r,1)/\sqrt{1 + r^2}$. This is a manifold with boundaries of the form
\[ z_1 = r/\sqrt{1 + r^2}, \ \ z_2  = 1/\sqrt{1 + r^2}.\]

\noindent Define $M = S^1_{+,+} \times X$ and we get the compactified structure.  

The formula \eqref{bpformula} then follows from integration by parts on the expression
\[
 \int_{C(X)}  [  f^{(1)} \overline{u^{(2)}} - u^{(1)} \overline{f^{(2)}} ] \chi(\ve r) dr dv_h
\]
for $\chi$ a smooth cut-off function localized near $0$.  Integrating by parts in $r$, applying formula \eqref{iosolns}, and taking the limit as $\ve \to 0$ the formula follows after an application of the Riemann-Lebesgue lemma.

Note, the boundary pairing formula is strongly related to structure of the Jost solutions through the existence of polyhomogeneous expansions (power series solutions near the boundary).

\subsubsection{Outline of the remaining arguments}

To prove the absence of embedded resonances, we use the boundary pairing formula with $f=0$ to prove that any outgoing solution vanishes to leading order on the $x_1=0$ boundary.  This may be iterated to show that in fact the power series vanishes to arbitrary order and hence the solution is indeed Schwarz on the cone.  To see this, differentiate the equation with respect to $x_1$ and look at the resulting inhomogeneous equation in the boundary pairing.  Otherwise, the power series solution depends uniquely on the first terms in the expansion.

In most circumstances, eliminating embedded eigenvalues requires unique continuation.  In the radial problem, we may apply the Theorem XIII.56 of \cite{RSv4} built around Jost solutions as mentioned above.  For conic metrics, we must follow the procedure of Froese-Herbst \cite{froese1982absence} built around seperable metric, which has been extended and formalized by Vasy \cite{vasy1997propagation,vasy2004exponential}.

\section{Construction of the Free Resolvent}
\label{app:free_res}
\noindent In this appendix, we provide a detailed construction integral kernel for the free resolvent operator 
\begin{equation}
R_0(z^2) = (-\Delta_{C(X)} - z^2)^{-1}: L^2(C(X))\to L^2(C(X)),
\end{equation}
for $\Im z \ne 0,$ closely following the exposition of \cite{baskin2019scattering}. This is equivalent to analyzing solutions of the equation
\begin{equation}\label{free_resolvent}
(-\Delta_{C(X)} - z^2)u = f
\end{equation}
for $f\in L^2(C(X)).$ To proceed, we decompose $u$ and $f$ into the basis $\{\varphi_j\}$ of eigenfunctions on $X$ as
$$f(r,\theta) = \sum\limits_{j=1}^\infty f_j(r)\varphi_j(\theta), \quad u(r,\theta) = \sum\limits_{j=1}^\infty u_j(r)\varphi_j(\theta).$$ 
Denote by $-\mu_j^2$ the eigenvalues of $\Delta_h$ associated to each $\varphi_j$. Then, we obtain that \eqref{free_resolvent} is equivalent to the collection of equations
\begin{equation}\lp\partial_r^2 + \frac{n-1}{r}\partial_r +z^2 - \frac{\mu_j^2}{r^2}\rp u_j(r) = -f_j(r) ,\hskip 0.2in j = 0,1,2,\dotsc.\label{jth_resolvent_eqn}
\end{equation}
Therefore, we can express the resolvent $R_0(z^2)$ as 
\[R_0(z^2)f(r,\theta) = \sum\limits_{j=0}^\infty u_j(r)\varphi_j(\theta),\]
with $u_j$ as above. If we define the $j$th \textit{radial resolvent} $R_{0,j}(z^2)$ by
\begin{equation}\label{radial_resolvents}
R_{0,j}(z^2) = \lp\partial_r^2 + \frac{n-1}{r}\partial_r + z^2 - \frac{\mu_j^2}{r^2}\rp^{-1}
\end{equation}
as an operator on $L^2(\R^+,r^{n-1}\,dr)$, then the full resolvent is given by
\[R_0(z^2)f(r,\theta) = \sum\limits_{j=0}^\infty R_{0,j}(z^2)f_j(r)\varphi_j(\theta).\]
For each $j$, the defining equation \eqref{jth_resolvent_eqn} for $R_{0,j}(z^2)f_j$ is an ODE with a regular singular point at zero, and so by applying the Frobenius method we find that the indicial roots of the equation are $-\frac{n-2}{2} \pm \sqrt{\lp\frac{n-2}{2}\rp^2 + \mu_j^2}.$ For this reason, we introduce the notations $\delta = -\frac{n-2}{2}$ and $\nu_j = \sqrt{\lp\frac{n-2}{2}\rp^2  +\mu_j^2} $. The structure of the indicial roots suggests that we rescale by $r^{\delta}$, and so we define $\omega_j$ by $u_j(r) = r^{\delta}\omega_j(r)$ so that $\omega_j$ is analytic near $r = 0.$ Then, \eqref{jth_resolvent_eqn} becomes
\[\partial_r^2\omega_j + \frac{1}{r}\partial_r\omega_j  + \left( z^2 - \frac{\nu_j^2}{r^2}\right)\omega_j = -r^{-\delta}f_j(r), \hskip 0.4in j = 0,1,2,\dotsc.\]
At this point it is helpful to restrict to particular class of $f_j$, namely those for which the Fourier transform $\widehat f_j$ is compactly supported (in order to compute the Fourier transform, we simply extend $f_j$ by zero to a function on all of $\R$). For such $f_j$, we know that there exists a holomorphic extension to all of $\C$ by the Paley-Weiner-Schwartz Theorem. We continue to denote this extension by $f_j.$ That we can make this restriction without loss of generality follows from the fact that such functions are dense in $L^2.$ Given this, if $z\ne 0$, we make the change variables via $\zeta = zr$ to obtain the following inhomogeneous Bessel equation of order $\nu_j:$
\begin{equation}\label{Bessel_eqn}
\wt\omega_j'' + \frac{1}{\zeta}\wt\omega_j'  + \left( 1 - \frac{\nu_j^2}{\zeta^2}\right)\wt\omega_j = -\frac{\zeta^{-\delta}}{z^2}f_j(\zeta/z),
\end{equation}
where $\wt\omega_j(\zeta) = \omega_j(\zeta/z)$, and the ``prime" notation denotes the complex derivative with respect to $\zeta$. Here, we define $\zeta^{-\delta}$ using the principal branch of the square root. For notational convenience, we define $f_{j,z}(\zeta):= -\frac{\zeta^{-\delta}}{z^2}f_j(\zeta/z)$, which is holomorphic for ${\zeta\in\C\setminus(-\infty,0].}$

The solutions to the homogeneous Bessel equation of order $\nu$ are the well-known Bessel functions of the first and second kind, denoted $J_{\nu}$ and $Y_{\nu}$, respectively. Closely related to these are the Hankel functions $H_{\nu}^{(1)}$ and $H_{\nu}^{(2)}$, given by 
\[\Hone = J_{\nu} + i Y_\nu, \hskip 0.3in \Htwo = J_\nu - i Y_\nu.\]
Any two of these Bessel and/or Hankel functions can be used to form a fundamental solution set for the homogeneous equation. Given an appropriate choice of fundamental solution set, we use the method of variation of parameters to construct solutions to the inhomogeneous problem. So let $y_1,y_2$ be a fundamental solution set for the homogeneous problem associated to \eqref{Bessel_eqn} for some fixed $j$. We then construct our solution $\wt\omega_j$ as 
\[\wt\omega_j = v_1y_1 + v_2y_2\]
where all objects above are functions of $\zeta$. Straightforward calculations show that if
\[v_1'(\zeta) = -\frac{y_2(\zeta)f_{j,z}(\zeta)}{\mathscr W(y_1,y_2)(\zeta)}, \hskip 0.3in \text{and }\hskip 0.3in v_2'(\zeta) =  \frac{y_1(\zeta)f_{j,z}(\zeta)}{\mathscr W(y_1,y_2)(\zeta)},\]
then $\wt\omega_j$ as given above solves the inhomogeneous equation \eqref{Bessel_eqn}, where $\mathscr W(y_1,y_2)(\zeta)$ denotes the Wronskian determinant of $y_1$ and $y_2$ evaluated at $\zeta$. Therefore, we may compute $v_1$ and $v_2$ by taking path integrals in the complex plane, which yields 
\[\wt\omega_j(\zeta) = \lp\int_{\mathscr C_1(\zeta)} -\frac{y_2(\xi)f_{j,z}(\xi)}{\mathscr W(y_1,y_2)(\xi)}\,d\xi\rp y_1(\zeta) + \lp\int_{\mathscr C_2(\zeta)} \frac{y_1(\xi)f_{j,z}(\xi)}{\mathscr W(y_1,y_2)(\xi)}\,d\xi\rp y_2(\zeta)\]
where $\mathscr C_1(\zeta)$, $\mathscr C_2(\zeta)$ are any complex contours connecting fixed points $c_1,c_2\in \C\setminus(-\infty,0]$ to $\zeta$, respectively. In fact, it suffices to take $c_1,c_2\in\R^+$. We then choose our contours to be the piecewise linear paths defined by 
\[\mathscr C_1(\zeta) = \{(1-t)c_1 + t \Re \zeta: t\in[0,1]\}\cup \{\Re\zeta + it\Im\zeta : t\in [0,1]\}\]
and 
\[\mathscr C_2(\zeta) = \{(1-t)c_2 + t\Re \zeta:t\in [0,1]\}\cup \{\Re \zeta + it\Im\zeta : t\in [0,1]\}.\]
Of particular interest are the boundary values of the resolvent near the continuous spectrum of $-\Delta_{C(X)}+V$. Therefore, if we consider $z^2 = \lambda^2 \pm i\ve$, we have 
\begin{align*}
\wt\omega_j(zr) & = y_1(zr)\lp\int\limits_{c_1}^{\lambda r}\frac{-y_2(t)f_{j,z}(t)}{\mathscr W(y_1,y_2)(t)}\,dt  + i\int\limits_0^{\pm \ve r} \frac{-y_2(\lambda r + it)f_{j,z}(\lambda r + it)}{\mathscr W(y_1,y_2)(\lambda r+ it)}\,dt\rp \\
& \hskip 0.5in + y_2(zr) \lp\int\limits_{c_2}^{\lambda r}\frac{y_1(t)f_{j,z}(t)}{\mathscr W(y_1,y_2)(t)}\,dt + i\int\limits_0^{\pm \ve r}\frac{y_1(\lambda r + it)f_{j,z}(\lambda r + it)}{\mathscr W(y_1,y_2)(\lambda r + it)}\,dt\rp 
\end{align*}
All that remains is to determine that appropriate fundamental solution set $y_1$, $y_2$ and constants $c_1,c_2$ so that our solution is a well defined element of $L^2(C(X))$. If we take $y_2 = J_{\nu_j}$ and $c_1 = 0,$ then $\wt\omega_j$ is bounded as $r \to 0$, provided that the coefficient integrals converge. We then choose $y_1$ to be either $H_{\nu_j}^{(1)}$ or $H_{\nu_j}^{(2)}$, depending on the sign of $\Im z$. By the asymptotic forms of the Hankel functions, we have 
\[H_{\nu_j}^{(1)}(\zeta) \sim \sqrt{\frac{2}{\pi \zeta}} e^{i\lp \zeta - \frac{\nu_j\pi}{2} - \frac{\pi}{4}\rp}\]
and 
\[H_{\nu_j}^{(2)}(\zeta) \sim \sqrt{\frac{2}{\pi \zeta}} e^{-i \lp \zeta - \frac{\nu_j\pi}{2} - \frac{\pi}{4}\rp }\]
for $-\pi < \arg \zeta < \pi$, and the branch of the square root is defined by $\zeta^{1/2} = e^{\frac{1}{2}(\ln|\zeta| + i \arg \zeta)}$ for such $\zeta.$ We can now see that if $z^2 = \lambda^2 + i\ve$, then $\zeta = z r$ also has positive imaginary part, and so $H_{\nu_j}^{(1)}(zr)$ decays exponentially as $r\to \infty$, while $H_{\nu_j}^{(2)}$ exhibits exponential growth. Hence, when $z^2 = \lambda^2 + i\ve$ we take $y_1 = H_{\nu_j}^{(1)}$ and $c_2 = \infty$, which yields 
\begin{align*}
 \wt\omega_j(zr) & = H_{\nu_j}^{(1)}(zr)\lp\int\limits_{0}^{\lambda r}\frac{J_{\nu_j}(t)f_{j,z}(t)}{2i/(\pi t)}\,dt  + i\int\limits_0^{\pm \ve r} \frac{J_{\nu_j}(\lambda r + it)f_{j,z}(\lambda r + it)}{2i/[\pi(\lambda r+ it)]}\,dt\rp \\
& \hskip 0.5in + J_{\nu_j}(zr) \lp\int\limits_{\lambda r}^{\infty}\frac{H_{\nu_j}^{(1)}(t)f_{j,z}(t)}{2i/(\pi t)}\,dt - i\int\limits_0^{\pm \ve r}\frac{H_{\nu_j}^{(1)}(\lambda r + it)f_{j,z}(\lambda r + it)}{2i/[\pi(\lambda r +it )]}\,dt\rp,
\end{align*}
since $\mathscr W(H_{\nu_j}^{(1)},J_{\nu_j})(\xi) = -\frac{2i}{\pi \xi}$. We can then take the limit as $\ve \to 0$ to obtain
\[\wt\omega_j(\lambda r) =\frac{\pi}{2i}H_{\nu_j}^{(1)}(zr)\int\limits_{0}^{\lambda r}tJ_{\nu_j}(t)f_{j,z}(t)\,dt   + \frac{\pi}{2i}J_{\nu_j}(zr) \int\limits_{\lambda r}^{\infty}tH_{\nu_j}^{(1)}(t)f_{j,z}(t)\,dt.\]
Recalling that $u_j(r) = (zr)^\delta\wt\omega_j(zr)$ and $f_{j,z}(t) = -\frac{t^{\frac{n-2}{2}}}{z^2}f_j(t/z)$, we get that the outgoing solution corresponding to the $j$th resolvent is
\begin{align*}
  &  u_j(r) =  \frac{\pi i}{2}(\lambda r)^{-\frac{n-2}{2}}H_{\nu_j}^{(1)}(\lambda r)\int\limits_{0}^{\lambda r}\frac{t^{\frac{n}{2}}J_{\nu_j}(t)f_{j}(t/\lambda)}{\lambda^2}\,dt  \\
  & \hspace{2cm} + \frac{\pi i}{2}(\lambda r)^{-\frac{n-2}{2}}J_{\nu_j}(\lambda r) \int\limits_{\lambda r}^{\infty}\frac{t^{\frac{n}{2}}H_{\nu_j}^{(1)}(t)f_{j}(t/\lambda)}{\lambda^2}\,dt.
    \end{align*}
    If we then change variables via $t = \lambda s$, we can rewrite the above as
\[u_j( r) = \frac{\pi i}{2} r^{-\frac{n-2}{2}}H_{\nu_j}^{(1)}(\lambda r) \int\limits_0^{r} s^{\frac{n}{2}}J_{\nu_j}(\lambda s) f_j(s)\,ds + \frac{\pi i}{2}r^{-\frac{n-2}{2}}J_{\nu_j}(\lambda r)\int\limits_{r}^\infty s^{\frac{n}{2}}H_{\nu_j}^{(1)}(\lambda s)f_j(s)\,ds.\]
The integral kernel of $R_{0,j}(\lambda^2+i0)$ with respect to the measure $s^{n-1}\,ds$ is therefore given by 
\begin{equation}
\label{radial_resolvents_plus}R_{0,j}(\lambda^2+i0;r,s) = \begin{cases}
\frac{\pi i}{2}(rs)^{-\frac{n-2}{2}} J_{\nu_j}(\lambda s)H_{\nu_j}^{(1)}(\lambda r), & s < r\\
\frac{\pi i}{2}(rs)^{-\frac{n-2}{2}}J_{\nu_j}(\lambda r)H_{\nu_j}^{(1)}(\lambda s), & s > r,
\end{cases}
\end{equation}
since $s^{\frac{n}{2}} = s^{n-1}s^{-\frac{n-2}{2}}$.

We can repeat this analysis for $z^2 = \lambda^2 - i\ve$, and we find that we must take use $H_{\nu_j}^{(2)}$ instead of $H_{\nu_j}^{(1)}$ due to the asymptotic behavior at infinity, which also causes the Wronskian to change sign, but otherwise the calculations are identical. We therefore obtain
\begin{equation}\label{radial_resolvents_minus}
R_{0,j}(\lambda^2 - i0;r,s) = \begin{cases}
\frac{\pi }{2i}(rs)^{-\frac{n-2}{2}}J_{\nu_j}(\lambda s) H_{\nu_j}^{(2)}(\lambda r), & s < r\\
\frac{\pi }{2i}(rs)^{-\frac{n-2}{2}}J_{\nu_j}(\lambda r)H_{\nu_j}^{(2)}(\lambda s), & s > r.
\end{cases}.
\end{equation}

\begin{rmk}\textnormal{
We note that one could also obtain the formula for $R_{0,j}(\lambda^2 -i0)$ from that of $R_{0,j}(\lambda^2 + i0)$ by using the analytic continuation formulae
\[J_\nu(ze^{\pi i}) = e^{\nu\pi i}J_\nu(z)\quad \text{and}\quad H_{\nu}^{(1)}(ze^{-\pi i}) = -e^{-\nu \pi i}H_{\nu}^{(2)}(z).\]
}
\end{rmk}

Given \eqref{radial_resolvents_plus} and \eqref{radial_resolvents_minus}, we can express the imaginary part of the resolvent kernels $R_{0,j}$ as follows. 
\begin{lemma}\label{stone}
For $\lambda$ real, we have
\[\Imag R_{0,j}(\lambda^2+i0;r,s)= \frac{\pi }{2}(rs)^{-\frac{n-2}{2}}J_{\nu_j}(\lambda r)J_{\nu_j}(\lambda s)\]
as an integral kernel with respect to the measure $s^{n-1}ds$.
\end{lemma}

\begin{proof}
This follows immediately from the fact that 
\[H_{\nu_j}^{(1)} + H_{\nu_j}^{(2)} = (J_{\nu_j} + iY_{\nu_j}) + (J_{\nu_j} -i Y_{\nu_j}) = 2J_{\nu_j}.\]
\end{proof}

We can now write down an expression for the spectral measure of $-\Delta_{C(X)}$ as in \cite{Cheeger1979}, which follows from Stone's formula.

\begin{lemma}
For $\lambda$ real, 
\[\Imag R_{0}(\lambda^2+i0;x,y) =  \frac{\pi}{2}(rs)^{-\frac{n-2}{2}}\sum\limits_{j=0}^\infty J_{\nu_j}(\lambda r)J_{\nu_j}(\lambda s)\varphi_j(\theta_1)\overline{\varphi_j(\theta_2)} \]
where $x = (r,\theta_1)$ and $y = (s,\theta_2)$ are points in $C(X)$. Moreover, the absolutely continuous part of the spectral measure of $-\Delta_{C(X)}$, with the convention that $\lambda^2$ is the spectral parameter, is given by
\begin{align*}
d\Pi_0(\lambda;x,y) & = \frac{1}{2\pi i}\left[R_0(\lambda^2+i0;x,y)- R_0(\lambda^2 - i0;x,y)\right]2\lambda\,d\lambda \\
&= \sum\limits_{j=0}^\infty (rs)^{-\frac{n-2}{2}}J_{\nu_j}(\lambda r)J_{\nu_j}(\lambda s)\varphi_j(\theta_1)\overline{\varphi_j(\theta_2)}\lambda\,d\lambda.
\end{align*}
\end{lemma}

\section{Dispersive estimates for the free Schr\"odinger equation}
\label{ell2_dispersive}

Here, we discuss bounds on solutions of the unperturbed Schr\"odinger equation, given by
\begin{equation} \label{free_schrodinger} 
\begin{cases}
\lp \frac{1}{i}\partial_t  - \Delta_{C(X)}\rp u  = 0, \\
u|_{t=0} = f.
\end{cases}
\end{equation}
We prove that the solution to this equation satisfies a dispersive estimate analogous to \thmref{disp_est}, but without the need for projection onto the harmonics of the link, provided that the solution is measured in $L^\infty(\R^+;L^2(X))$, rather than simply $L^\infty (C(X))$. We do this by using a modification of the techniques outlined in \cite{Ford2010}, which handled flat two-dimensional cones, to obtain an explicit asymptotic formula for the kernel of $e^{it\Delta_{C(X)}}$ as a function of a rescaled variable.

\begin{theorem}\label{free_dispersive_bound}
Let $C(X) = \R^+\times X$ be the product cone on $X$, for $(X,h)$ a compact Riemannian manifold of dimension $n-1$. Then the solution to \eqref{free_schrodinger} satisfies
\[\|e^{it\Delta_{C(X)}}f\|_{L^{\infty}\lp\R^+; L^2(X)\rp} \le C t^{-\frac{n}{2}}\|f\|_{L^1(\R^+;L^2(X))}, \hskip 0.2in t > 0,\]
for some $C> 0.$ Here, $L^1(\R^+)$ is defined with respect to the measure $r^{n-1}\,dr.$
\end{theorem}

\begin{rmk}\textnormal{
We note that this result is somewhat weaker than similar estimates obtained in \cite{zhang2018strichartz}, but we include it here because the proof is quite short and requires significantly less machinery. 
}
\end{rmk}

 Since $X$ is compact, there exists an orthonormal basis $\{\varphi_j\}_{j=0}^\infty$ of $L^2(X)$, satisfying 
 \[-\Delta_h\varphi_j = \mu_j^2\varphi_j\] 
 for $0 = \mu_0^2 < \mu_1^2 \le \dotsc$ repeated according to multiplicity. By the functional calculus of Cheeger \cite{Cheeger1979} discussed in Section \ref{free}, we can define the shifted eigenvalues 
\[\nu_j = \sqrt{\mu_j^2 + \lp\frac{n-2}{2}\rp^2},\]
in order to write the spectral measure of $-\Delta_{C(X)}$ as
\[d\Pi_0(r_1,\theta_1,r_2,\theta_2) = (r_1r_2)^{-\frac{n-2}{2}}\sum\limits_{j=0}^\infty  J_{\nu_j}(\lambda r_1)J_{\nu_j}(\lambda r_2)\varphi_j(\theta_1)\overline{\varphi_j(\theta_2)}\,\lambda\,d\lambda,\]
where $J_\nu$ is the Bessel function of the first kind of order $\nu.$ Hence, the fundamental solution to \eqref{free_schrodinger} has the form
\begin{equation}\label{fund_solution}
K_{e^{it\Delta_{C(X)}}}(r_1,\theta_1,r_2,\theta_2) = (r_1r_2)^{-\lp\frac{n-2}{2}\rp}\sum\limits_{j=0}^\infty\lp \int\limits_0^\infty e^{it\lambda^2}J_{\nu_j}(\lambda r_1)J_\nu(\lambda r_2)\lambda\,d\lambda\rp\varphi_j(\theta_1)\overline{\varphi_j(\theta_2)},
\end{equation}
with respect to the standard measure on the cone, $r^{n-1}\,dr\,dv_h(\theta)$, where $dv_h$ is the Riemannian volume measure on $X$. As in \cite{Ford2010}, we let $t = is$ in the above expression to obtain a formula for the heat kernel $e^{-s\Delta_{C(X)}}$. By Weber's second exponential integral formula, we have that
\[\int\limits_0^\infty e^{-s\lambda^2} J_\nu(\lambda r_1)J_\nu(\lambda r_2)\lambda\,d\lambda = \frac{1}{2s}e^{-\frac{r_1^2 + r_2^2}{4s}}I_\nu\lp\frac{r_1r_2}{2s}\rp,\]
where $I_\nu$ is the modified Bessel function of order $\nu,$ defined by 
\[I_\nu(x) = \sum\limits_{k=0}^\infty \frac{1}{k!\,\Gamma(\nu+k+1)}\lp\frac{x}{2}\rp^{2k+\nu}.\]
Analytic continuation in $s$ and taking $s = -it$ gives us
\[K_{e^{it\Delta_{C(X)}}}(r_1,\theta_1,r_2,\theta_2) = \frac{ie^{\frac{r_1^2 + r_2^2}{4it}}}{2t(r_1r_2)^{\frac{n-2}{2}}}\sum\limits_{j=0}^\infty i^{\nu_j}J_{\nu_j}\lp\frac{r_1r_2}{2t}\rp\varphi_j(\theta_1)\overline{\varphi_j(\theta_2)},\]
since $I_\nu(ix) = i^\nu J_\nu(x)$. For non-integer values of $\nu,$ we choose $z^\nu$ to have its branch cut along the negative real axis. For convenience, we define $x = \frac{r_1r_2}{2t}$ and let
\[S(x,\theta_1,\theta_2) = x^{-\lp\frac{n-2}{2}\rp}\sum\limits_{j=0}^\infty i^{\nu_j}J_{\nu_j}\lp x\rp\varphi_j(\theta_1)\overline{\varphi_j(\theta_2)},\]
so that 
\begin{equation}\label{K_S_relation}
K_{e^{it\Delta_{C(X)}}}(r_1,\theta_1,r_2,\theta_2) = \frac{i\exp\lp\frac{r_1^2 + r_2^2}{4it}\rp}{(2t)^{\frac{n}{2}}}S(x,\theta_1,\theta_2). 
\end{equation}
Furthermore, we define the family of operators $S(x):C^\infty(X)\to \mathcal D'(X)$ by
\[S(x)f(\theta_1) = \int\limits_X S(x,\theta_1,\theta_2)f(\theta_2)\,dv_h(\theta_2) = x^{-\lp\frac{n-2}{2}\rp}\sum\limits_{j=0}^\infty i^{\nu_j}J_{\nu_j}(x)\langle f,\varphi_j\rangle\varphi_j(\theta).\]

Next, we make note of an asymptotic expansion for $K_{e^{it\Delta_{C(X)}}}$ in the regime where $x\to 0$, which is analogous to \cite[Prop 4.1]{Ford2010}. 

\begin{proposition}\label{small_x_asymptotics}
The free Schr\"odinger propagator has the asymptotic behavior
\[K_{e^{it\Delta_{C(X)}}}(r_1,\theta_1,r_2,\theta_2) = \frac{i\exp\lp\frac{r_1^2 + r_2^2}{4it}\rp}{(2t)^{\frac{n}{2}}}\left[ \frac{(i/2)^{\frac{n-2}{2}}}{\Gamma(\frac{n}{2})\textnormal{vol}(X)} + \mathcal O\lp\lp\frac{r_1r_2}{2t}\rp^\alpha\rp\right],\hskip 0.2in \text{as }\frac{r_1r_2}{2t}\to 0,\]
where $\alpha = \min\{2,\nu_1 - \frac{n-2}{2}\}$.
\end{proposition}

\begin{proof}
It suffices to show that 
\begin{equation}\label{S_asymptotics}
S(x,\theta_1,\theta_2) = \frac{(i/2)^{\frac{n-2}{2}}}{\Gamma(\frac{n}{2})\textnormal{vol}(X)}  + \mathcal O\lp x^{\alpha}\rp,\hskip 0.2in \text{as }x\to 0,
\end{equation}
uniformly in $\theta_1,\theta_2.$ Since the $\varphi_j$ are $L^2$-normalized, we have that $\varphi_0 = \frac{1}{\sqrt{\text{vol}(X)}}$, and so
\[S(x,\theta_1,\theta_2) = \frac{(i/2)^{\frac{n-2}{2}}}{\text{vol}(X)}\lp\frac{x}{2}\rp^{-\lp\frac{n-2}{2}\rp}J_{\frac{n-2}{2}}(x) +  x^{-\lp\frac{n-2}{2}\rp}\sum\limits_{j=1}^\infty i^{\nu_j}J_{\nu_j}(x)\varphi_j(\theta_1)\overline{\varphi_j(\theta_2)}.\]
By the standard power series representation for $J_{\frac{n-2}{2}}$, we have
\begin{align*}
& \left|S(x,\theta_1,\theta_2) - \frac{(i/2)^{\frac{n-2}{2}}}{\Gamma(\frac{n}{2})\textnormal{vol}(X)}\right|  = \\
& \hspace{.5cm} \left|\frac{(i/2)^{\frac{n-2}{2}}}{\text{vol}(X)}\sum\limits_{k=1}^\infty \frac{(-1)^k}{k!\,\Gamma(\frac{n}{2}+k)}\lp\frac{x}{2}\rp^{2k} + x^{-\lp\frac{n-2}{2}\rp}\sum\limits_{j=1}^\infty i^{\nu_j}J_{\nu_j}(x)\varphi_j(\theta_1)\overline{\varphi_j(\theta_2)}\right|.\\
\end{align*}
Using the bound $\left|J_\nu(x)\right|\le \frac{1}{\Gamma(\nu+1)}\lp\frac{x}{2}\rp^\nu$ for $x$ real and $\nu > 0$, along with the standard $L^\infty$ eigenfunction estimate $\|\varphi_j\|_{L^{\infty}} \le C \mu_j^{\frac{n-2}{2}}$, shows that for $0\le x < 2,$ we have
\begin{align}
\begin{split}\label{S_kernel_bound}
\left|S(x,\theta_1,\theta_2) - \frac{(i/2)^{\frac{n-2}{2}}}{\Gamma(\frac{n}{2})\textnormal{vol}(X)}\right| & \le \frac{1}{\text{vol}(X)}\sum\limits_{k=1}^\infty \lp\frac{x^2}{4}\rp^k + \sum\limits_{j=1}^\infty \frac{\mu_j^{2(n-1)}}{2^{\frac{n-2}{2}}\Gamma(\nu_j+1)}\lp\frac{x}{2}\rp^{\nu_j-\frac{n-2}{2}}\\
& \le \frac{x^2}{\text{vol}(X)(4-x^2)} + \lp\frac{x}{2}\rp^{\nu_1-\frac{n-2}{2}}\sum\limits_{j=1}^\infty\frac{\mu_j^{2(n-1)}}{\Gamma(\nu_j+1)}.
\end{split}
\end{align}
Note that $\mu_j\sim Cj^{\frac{1}{n-1}}$ by the Weyl law for the eigenvalues of $-\Delta_h$. Since $\nu_j \ge \mu_j$ for all $j,$ the summation in the last inequality of \eqref{S_kernel_bound} converges, which demonstrates \eqref{S_asymptotics} with $\alpha = \min\{2,\nu_1-\frac{n-2}{2}\}.$ We observe that by definition, $\nu_1 - \frac{n-2}{2} > 0.$ Hence, by \eqref{K_S_relation}, the proof is complete. 
\end{proof}

\begin{cor}\label{S_L2_bound}
The family of operators $S(x)$ satisfies
\[\|S(x)f\|_{L^2(X)} \le C\|f\|_{L^2(X)}\]
for all $x \ge 0$ and for some $C> 0$ which is uniform in $x$.
\end{cor}

\begin{proof}
For $x < 2- \ve$ with $\ve > 0$, the estimate follows from the proof of Proposition \ref{small_x_asymptotics}, which shows that $|S(x,\theta_1,\theta_2)|\le C$ for some $C$ which is uniform in $\theta_1,\theta_2$. Thus, for such $x$,
\begin{align*}\|S(x)f\|_{L^2(X)}^2 & = \int\limits_X \left|\int\limits_X S(x,\theta_1,\theta_2)f(\theta_2)\,dv_h(\theta_2)\right|^2\,dv_h(\theta_1)\\
& \le \int\limits_X \|S(x,\theta_1,\cdot)\|_{L^2(X)}^2 \|f\|_{L^2(X)}^2\,dv_h(\theta_1)\\
& \le C^2\text{vol(X)}^2\|f\|_{L^2(X)}^2.
\end{align*}
For $x \ge 1,$ we can simply use the fact that $|J_\nu(x)|$ is bounded uniformly in both $x$ and $\nu$ to see that 
\[\|S(x)f\|_{L^2(X)}^2 = x^{2-n}\sum\limits_{j=0}^\infty J_{\nu_j}^2(x)\langle f,\varphi_j\rangle^2 \le C\sum\limits_{j=0}^\infty \langle f,\varphi_j\rangle^2 = C\|f\|_{L^2(X)}^2,\]
for some $C>0$, since $n \ge 2.$ 
\end{proof}

We are now ready to present the proof of the dispersive estimate in Theorem \ref{free_dispersive_bound}.
\begin{proof}[Proof of Theorem \ref{free_dispersive_bound}]
Recalling that $K_{e^{it\Delta_{C(X)}}}(r_1,\theta_1,r_2,\theta_2)$ is the Schwartz kernel of $e^{it\Delta_{C(X)}}$ and applying $\eqref{K_S_relation}$, we have that for any $r_1,t > 0$,
\begin{align*}
&\|e^{it\Delta_{C(X)}}f(r_1,\cdot)\|_{L^2(X)}^2 \\
&= \int\limits_X\left|\int\limits_0^\infty\int\limits_X K_{e^{it\Delta_{C(X)}}}(r_1,\theta_1,r_2,\theta_2) f(r_2,\theta_2)r_2^{n-1}\,dv_h(\theta_2)\,dr_2\right|^2\,dv_h(\theta_1)\\
& = \int\limits_X \lp\int\limits_0^\infty \frac{ie^{\lp\frac{r_1^2+r_2^2}{4it}\rp}}{(2t)^{\frac{n}{2}}}\left[S\lp\frac{r_1r_2}{2t}\rp f(r_2,\theta_1)\right] r_2^{n-1}\,dr_2\rp  \\
& \hspace{1.5cm}\times \overline{\lp\int\limits_0^\infty \frac{ie^{\lp\frac{r_1^2+r_3^2}{4it}\rp}}{(2t)^{\frac{n}{2}}}\left[S\lp\frac{r_1r_3}{2t}\rp f(r_3,\theta_1)\right] r_3^{n-1}\,dr_3\rp}\,dv_h(\theta_1)\\
&= (2t)^{-n}\int\limits_0^\infty\int\limits_0^\infty e^{\lp\frac{r_2^2-r_3^2}{4it}\rp}\lp \int\limits_X S\lp\frac{r_1r_2}{2t}\rp f(r_2,\theta_1)\overline{S\lp\frac{r_2r_3}{2t}\rp f(r_3,\theta_1)}\,dv_h(\theta_1)\rp\,(r_2r_3)^{n-1} \,dr_2\,dr_3\\
&  \le (2t)^{-n}\int\limits_0^\infty\int\limits_0^\infty \left\|S\lp\frac{r_1r_2}{2t}\rp f(r_2,\cdot)\right\|_{L^2(X)}\left\|S\lp\frac{r_1r_3}{2t}\rp f(r_3,\cdot)\right\|_{L^2(X)}\,(r_2r_3)^{n-1} \,dr_2\,dr_3.\\
\end{align*}
In the last inequality, we are able to omit the complex exponential factor by taking absolute values, since the integral is known to be real-valued and non-negative. By Corollary \ref{S_L2_bound}, the above is bounded by 
\[C t^{-n}\int\limits_0^\infty\int\limits_0^\infty \|f(r_2,\cdot)\|_{L^2(X)}\|f(r_3,\cdot)\|_{L^2(X)}\,(r_2r_3)^{n-1}\,dr_2\,dr_3 = Ct^{-n}\|f\|_{L^1(L^2(X),r^{n-1}\,dr)}^2, \]
for some $C>0$ which is independent of $r_1$. Taking square roots completes the proof of \thmref{free_dispersive_bound}.
\end{proof}

\bibliographystyle{abbrv}
\bibliography{MMT-bib1}

\end{document}